\documentclass[10pt]{amsart}
\usepackage[margin=1.5in]{geometry}
\usepackage{color}
\usepackage{amssymb}
\usepackage{mathtools}
\newtheorem{theorem}{Theorem}[section]
\newtheorem{lemma}[theorem]{Lemma}
\newtheorem{proposition}{Proposition}[section]

\theoremstyle{definition}

\theoremstyle{remark}

\numberwithin{equation}{section}

\usepackage{graphicx}
\DeclareGraphicsExtensions{.pdf,.png,.jpg}
\begin{document}
		\title[Relativistic BGK model for gas mixtures]{Relativistic BGK model for gas mixtures}

\author[Hwang]{Byung-Hoon Hwang}
\address{Department of Mathematics Education, Sangmyung University, 20 Hongjimun 2-gil, Jongno-Gu, Seoul 03016, Republic of Korea}
\email{bhhwang@smu.ac.kr}

\author[M-.S Lee ]{Myeong-Su Lee}
\address{Department of Mathematics, Sungkyunkwan University, Suwon 440-746, Republic of Korea}
\email{msl3573@skku.edu}

\author[S.-B. Yun]{Seok-Bae Yun}
\address{Department of Mathematics, Sungkyunkwan University, Suwon 440-746, Republic of Korea}
\email{sbyun01@skku.edu}


\begin{abstract}
Unlike the case for classical particles, the literature on BGK type models for relativistic gas mixture is extremely limited. 
There are a few results 
in which such relativistic BGK models for gas mixture are employed to compute transport coefficients. However,  to the best knowledge of authors, relativistic BGK models for gas mixtures with complete presentation of the relaxation operators are missing in the literature.
 In this paper, we fill this gap by suggesting a BGK model for relativistic gas mixtures for which the existence of each equilibrium coefficients in the relaxation operator is rigorously guaranteed in a way that all the essential physical properties are satisfied such as the conservation laws, the H-theorem, the capturing of the correct equilibrium state, the indifferentiability principle, and the recovery of the classical BGK model in the Newtonian limit.
\end{abstract}
\maketitle
\tableofcontents
\section{Introduction}
The BGK model \cite{BGK,W} is a kinetic equation originally designed to mimic the classical Boltzmann flows at a much lower computational cost. Due to the consistency with the Boltzmann equation, the BGK formulation has been widely developed to model the gas dynamics in various physical contexts. In the relativistic framework,
 such BGK type models were proposed by Marle \cite{Marle,Marle2}, and Anderson and Witting \cite{AW} for a single monatomic gas.
 In both models,the classcial Maxwellian is replaced with the relativistic Maxwellian (so-called the J\"{u}ttner distribution. See \cite{Juttner} ).
 The difference of these two models comes from how the four-velocity was chosen. In Marle's formulation, the four-velocity is regarded as a particle flow based on the Eckart frame \cite{Eckart}, whereas the four-velocity of the Anderson-Witting model is interpreted as an energy flow by the Landau-Lifshitz frame \cite{LL}. These models have been fruitfully used to understand various relativistic flow problems \cite{CKT,DHMNS,Kremer2,MBHS,MNR}. Mathematical investigation on these models started in the last decade \cite{BCNS,CJS,Hwang,HLY,HY,HY2}. Recently,  a new relativistic BGK model for
polyatomic particles in the Eckart frame was suggested in \cite{PR} based on the polyatomic generalization of the relativistic Maxwellian in \cite{PR2}
 A nice survey on this model can be found in \cite{ACPR,CPR,CPR2,HRY}. 


Gas mixture problem is important in the practical point of view, since virtually all physically relevant particle systems consist of several different species of particles. In classical framework, numerous BGK-type relaxation models have been proposed, which can be categorized into two classes according to whether the relaxation operator is given by a sum of bi-species relaxation operator or by  a single global relaxation operator. The first category began with the model by Gross and Krook \cite{GK} in 1956. After that, several models had been suggested by various researchers \cite{Greene,GSB,Hamel}. But none of these models guarantees all the basic properties of the Boltzmann equation such as the conservation laws, non-negativity of the distribution functions, H-theorem, indifferentiability and the correct equilibrium.  

This issue was addressed by Andries, Aoki, and Perthame \cite{AAP}, who pioneered the second category of models by employing a single BGK operator per species instead of a sum of bi-species operators. 
	Since then, many other models of the second category have been suggested \cite{B,BPS,GMS,TS}. 
Furthermore, due to the simplicity, the latter approach has been fruitfully adapted to extend the BGK model  to more general gas systems of various kinds, such as chemically reacting gases \cite{GRS,GS}, and polyatomic gas mixtures \cite{BMS,BM,BT}. 
	Recently, the interest on the BGK-type models of the first category has revived after a series of BGK-type models of the first category satisfying the conservation laws, non-negativity of the distribution functions, H-theorem have been suggested \cite{BBGSP,HHM,KPP,KPP2}.	We refer to \cite{BKPY}	for an  extension to quantum gas mixtures in 
	this direction.
	The existence theory of various BGK models for gas mixtures was addressed in \cite{BKPY2,KP2,Pirner}(first category), and \cite{KLY,KP2} (second category). For numerical simulations, see \cite{CBGR,CKP,GRS2,TVD}.  Review on BGK models for mixtures can be found in \cite{Pirner3,PS}.

Unlike the kinetic theories for classical particles, the related literature for BGK type models for relativistic gas mixture is extremely limited. There are a few papers \cite{Kremer,Kremer3,KP} in which relativistic BGK models for gas mixture are employed to compute transport coefficients. However,
to the best knowledge of authors, relativistic BGK models for gas mixtures with complete presentation of the relaxation operators are missing in the literature. The computation of transport coefficients is possible  even without the knowledge of how equilibrium coefficients in the relaxation operator are defined, but without such precise
definition, investigation of the dynamics of relativistic gases at the kinetic level using the model is impossible. 


In this regard, we suggest in this paper a consistent BGK model of the Marle-type for relativistic gas mixtures. 
Recall that the relativistic Boltzmann equation for gas mixtures reads
 \begin{align}\label{RB}
\partial_t f_i+\frac{cp_i}{p_i^0}\cdot\nabla_x f_i=Q_i,
\end{align}
where $f_i(x^{\mu}, p_i^{\mu})$ is the momentum distribution  representing the number density of $i$ species at the phase point $(x^{\mu}, p_i^{\mu})$. The right-hand side $Q_i$ is called the collision operator and given by
\begin{align*}
	Q_i=\sum_{j=1}^NQ_{ij}(f_i,f_j):=\sum_{j=1}^{N}\frac{1}{p^0_i}\int(f'_if'_j-f_if_j)F_{ij}\sigma_{ij}\,d\Omega\frac{dp_j}{p^0_j},
\end{align*}
where $F_{ij}$ is the invariant flux and $\sigma_{ij}d\Omega$ is the invariant elastic differential cross section for collision between species $i$ and $j$.  We replace this complicated collision operator with the following single relaxation operator:
\begin{align*}
	\widetilde{Q}_i=\frac{cm_i}{\tau_i p_i^0}(\mathcal{J}_i-f_i), \qquad i=1,\cdots,N.
\end{align*}
where the attractor $\mathcal{J}_i$ is given by the J\"{u}ttner distribution \cite{Juttner}
$$
\mathcal{J}_i=\frac{g_{s_{i}}}{h^3}\exp\left({\widetilde{\beta}\widetilde{\mu}_i-\widetilde{\beta}\widetilde{U}^\mu p_{i \mu }}\right)\qquad \text{with}\quad \widetilde{\beta}:=1/k\widetilde{T}
$$
where $h$ is the Planck constant, $k$ is the Boltzmann constant, and $g_{s_i}$ is the degeneracy factor of $i$ species with spin $s_i$:
$$
g_{s_i}=\begin{cases}
2s_i+1 & \text{for}\quad m\neq 0\cr
2s_i & \text{for}\quad m=0
\end{cases}.
$$
The equilibrium chemical potential $\widetilde{\mu}_i$, the equilibrium four-velocity $\widetilde{U}^\mu,$ and the equilibrium temperature $1/k\widetilde{\beta}$ are functions of $t$ and $x$, which are determined in a way that the conservation laws of each particle four-flow and total energy-momentum tensor are satisfied (for details, see Section 3). In this procedure, the equilibrium coefficient $\widetilde{\beta}$ is determined through the nonlinear relation in Proposition \ref{JF} (2):
$$
	\sum_{i=1}^{N}\frac{m_i}{\tau_i}\frac{\int_{\mathbb{R}^3} e^{-c\widetilde{\beta}p_i^0}\,dp_i}{\int_{\mathbb{R}^3} e^{-c\widetilde{\beta}p_i^0}\frac{dp_i}{p_i^0}}\int_{\mathbb{R}^3}f_i \,\frac{dp_i}{p_i^0}=\frac{1}{c}\left[\left( \sum_{i=1}^{N}\frac{m_i }{\tau_i}n_iU_i^\mu\right)\left( \sum_{j=1}^{N}\frac{m_j }{\tau_j}n_jU_{j\mu}\right)\right]^{\frac{1}{2}}.
$$
 Note that  this highly nonlinear relation inevitably leads to a non-trivial determination problem as to whether the equilibrium coefficients $\widetilde{\beta}$ is well-defined by the momentum distributions $f_i$. Such issue, however, has never been addressed for relativistic gas mixtures.  

In this regards, we prove in this paper the well-definedness of equilibrium coefficients $\widetilde{\mu}_i$, $\widetilde{U}^\mu$, and $\widetilde{\beta}$ using the constraint related to the conservation laws (see Theorem \ref{JF}). And we show that our model fulfills the
 fundamental kinetic properties: non-negativity of distribution functions, conservation laws, H-theorem, capturing of the correct equilibrium state, and indifferentiability principle. We also show that our model converges into a well-known classical multicomponent BGK model in the non-relativistic limit.\\

This paper is organized as follows. In Section 2, we introduce some preliminaries on
the relativistic kinetic theory. In Section 3, we briefly review important properties of the Boltzmann collision operator. Then, in Section 4, we present a relaxation operator for relativistic gas mixtures and show that the equilibrium coefficients constituting of the relaxation operator are uniquely determined solely by the constraints from the conservation laws.
In Section 5, we prove that our model satisfies the essential properties of the relativistic Boltzmann equation presented in Section 3. 
Finally, Section 6 is devoted to showing
that our model can recover the classical BGK model for gas mixtures
in the Newtonian limit.

\section{Preliminaries}
Throughout this paper, we consider relativistic gas mixtures of $N$ constituents where only elastic collisions occur. These gas mixtures are considered in a Minkowski spacetime with metric tensor $\eta_{\mu\nu}$ and its inverse $\eta^{\mu\nu}$ given by 
$$
\eta_{\mu\nu}=\eta^{\mu\nu}=\text{diag}(1,-1,-1,-1)
$$
where the Greek indices range from $0$ to $3$.    We use the raising and lowering indices
$$
\eta_{\mu\nu} a^\nu= a_\mu,\qquad \eta^{\mu\nu} a_\nu= a^\mu
$$ and follow the Einstein summation convention so that
$$
a^\mu b_\mu =\eta_{\mu\nu}a^\mu b^\nu= a^0 b^0-\sum_{j=1}^3a^j  b^j=	a_\mu b^\mu.
$$
 The particles of $i$ species are characterized by the space-time coordinates $x^{\mu}$ and the four-momentums $p_i^{\mu}$
 $$
 x^{\mu} =(ct,x)\in \mathbb{R}_+\times \Omega,\qquad p_i^{\mu}=(\sqrt{(cm_i)^2+|p_i|^2},p_i)\in\mathbb{R}_+\times\mathbb{R}^3.
 $$ 
Here $p_i^0$ is given by the mass-shell condition $ p_i^\mu p_{i \mu }=(cm_i)^2$, where $c$ is the speed of light and $m_i$ is the rest mass of $i$ species. Unlike the classical setting, the phase point $(x^\mu,p_i^\mu)$ has dependency on species $i$ since $p_i$ represents the microscopic momentum.

The function $f_i(x^{\mu}, p_i^{\mu})$ is the momentum distribution  representing the number density of $i$ species at the phase point $(x^{\mu}, p_i^{\mu})$. The partial particle four-flow $N_i^\mu$ and the partial energy-momentum tensor $T_i^{\mu\nu}$ are defined by
$$
N_i^\mu=c\int_{\mathbb{R}^3}p^\mu_if_i \,\frac{dp_i}{p_i^0},\qquad T_i^{\mu\nu}=c\int_{\mathbb{R}^3}p_i^\mu p_i^\nu f_i \frac{dp_i}{p_i^0}.
$$
The particle four-flow of gas mixtures $N^\mu$ and the energy-momentum tensor of gas mixtures $T^{\mu\nu}$ are obtained by
$$
N^\mu=\sum_{i=1}^{N}N_i^\mu, \qquad T^{\mu\nu}=\sum_{i=1}^{N}T_i^{\mu\nu}.
$$
According to the Eckart frame \cite{Eckart}, $N_i^\mu$ is decomposed into
\begin{equation}\label{eckart}
N_i^\mu= n_iU_i^\mu
\end{equation}
where $n_i$ is the (macroscopic) number density, and $U_i^\mu$ is the Eckart four-velocity
\begin{align*}\begin{split}
n_i&=\frac{1}{c}\int_{\mathbb{R}^3}p^\mu_i U_{i \mu }f_i \,\frac{dp_i}{p_i^0}=\Biggl\{\biggl(\int_{\mathbb{R}^3}f_i \,dp_i\biggl)^2-\sum_{j=1}^3\biggl(\int_{\mathbb{R}^3}p^j_if_i \,\frac{dp_i}{p_i^0}\biggl)^2\Biggl\}^{\frac{1}{2}},\cr
U_i^\mu&=\frac{c}{n_i}\int_{\mathbb{R}^3}p^\mu_if_i \,\frac{dp_i}{p_i^0}.
\end{split}\end{align*}
Here, $U^\mu_i$ has a constant length in the following sense
$$
U^\mu_i U_{i \mu }=c^2,\qquad \text{and hence}\qquad U_i^\mu=\left(\sqrt{c^2+|U_i|^2},U_i\right).
$$
The entropy four-flow of gas mixtures is defined as
\begin{equation*}\label{entropy}
S^\mu= -kc\sum_{i=1}^N\int_{\mathbb{R}^3}p_i^\mu  f_i \ln\left( \frac{f_i h^3}{g_{s_i}}\right) \,\frac{dp_i}{p_i^0}.
\end{equation*}

\section{Boltzmann collision operator for relativistic gas mixtures}
In this section, we briefly review the crucial properties of the Boltzmann collision operator for relativistic gas mixtures. These are the properties that the relaxation operator, to be constructed in the next section, is expected to satisfy.

\begin{enumerate}
	\item {\bf Conservation laws:} The Boltzmann collision operator $Q_i$ satisfies the following relations:
	\begin{align*}
		\begin{split}
			&\int_{\mathbb{R}^3} Q_i\, dp_i=0 \quad (i=1,\cdots,N), \qquad \sum_{i=1}^{N}\int_{\mathbb{R}^3} p_i^\mu Q_i\,dp_i=0,
		\end{split}
	\end{align*}
	which implies the following conservation laws for the partial particle four-flows $N_i^\mu$ and the energy-momentum tensor $T^{\mu\nu}$ of gas mixtures:
	$$
	\frac{\partial N_i^\mu}{\partial x^\mu}=0,\qquad \frac{\partial T^{\mu\nu}}{\partial x^\nu}=0.
	$$
	\item {\bf Equilibria:} In equilibrium, the momentum distribution $f_i$ is given by the J\"{u}ttner distribution which shares a common four-velocity and a common temperature. 
	That is,
	\begin{align*}
		Q_i=0,\ (\forall i=1,\cdots,N)\ \implies\ f_i=\frac{g_{s_{i}}}{h^3}\exp\left(\beta_E\mu_{E_{i}}\right)\exp\left(-\beta_E U_E^\mu p_{i \mu }\right)
	\end{align*}
	for some chemical potentials $\mu_{E_i}$ $(i=1,\cdots,N)$, four-velocity $U_E^\mu$, and temperature $1/k\beta_E$.
	\item {\bf H-theorem:} The following inequality holds:
	\begin{align*}
		\sum_{i=1}^N  \int_{\mathbb{R}^3} Q_i \ln f_i\, dp_i\leq0,
	\end{align*}
	which gives the H-theorem:
	$$
	\frac{\partial S^\mu}{\partial x^\mu} \ge 0.
	$$
	\item {\bf Indifferentiability principle:} When all the masses $m_i$ and the elastic differential cross sections $\sigma_{ij}$ are independent of the gas species, the following identity holds
	\begin{align*}
		\sum_{i=1}^N\sum_{j=1}^NQ_{ij}(f_i,f_j)=Q\left(\sum_{i=1}^Nf_i,\sum_{j=1}^Nf_j\right),
	\end{align*}
	where $Q$ denotes the relativistic collision operator for single-species gases, so that 
	the total distribution $f=\sum_{i=1}^Nf_i$ satisfies  the relativistic Boltzmann equation for single-species gases
	\begin{align*}
		\partial_t f+\frac{cp}{p^0}\cdot\nabla_x f=Q(f,f).
	\end{align*}
\end{enumerate}
For more details of the relativistic Boltzmann equation, see \cite{CK}.

\section{Construction of a relaxation model for relativistic gas mixtures}
In this section, we construct a relaxation model of the relativistic Boltzmann equation for gas mixtures that can reproduce all the properties presented in Section 3.
For this, we replace the Boltzmann collision operator in the right hand side of \eqref{RB} by a single Marle-type relaxation operator as follows:
\begin{align}\label{Marle}
\partial_t f_i+\frac{cp_i}{p_i^0}\cdot\nabla_x f_i=\widetilde{Q}_i:=\frac{cm_i}{\tau_i p_i^0}(\mathcal{J}_i-f_i) \qquad i=1,\cdots,N
\end{align}
where $\tau_i$ denotes the characteristic time of order of the time between collisions. 
In view of the property $(2)$ on the equilibrium state of the relativistic Boltzmann equation for gas mixture in the previous section, we choose  $\mathcal{J}_i$ of the following form:
\begin{align}\label{Juttner}
\mathcal{J}_i=\frac{g_{s_{i}}}{h^3}\exp\left(\widetilde{\beta}\widetilde{\mu}_i\right)\exp\left(-\widetilde{\beta}\widetilde{U}^\mu p_{i \mu }\right),
\end{align}
with equilibrium coefficients  $\widetilde{\mu}_i$ $(i=1,\cdots,N)$, $\widetilde{U}^\mu$ and $1/k\widetilde{\beta}$ representing the $i$-th equilibrium chemical potential, common four-velocity, and common temperature respectively. 
Here $\widetilde{U}^\mu$ has a constant length in the following sense:
$$
\widetilde{U}^\mu \widetilde{U}_\mu=c^2,\qquad \text{and hence} \qquad \widetilde{U}^0=\sqrt{c^2+ |\widetilde{U} |^2}.
$$ 
The form of $\mathcal{J}_i$ in \eqref{Juttner}  naturally drives the momentum distributions $f_i$ towards J\"{u}ttner distributions with a common Eckart four-velocity and a common temperature. Details can be found in Subsection 4.5. For our model to be consistent,  we choose the $N+4$ equilibrium coefficients $\widetilde{\mu}_i$ $(i=1,\cdots,N)$, $\widetilde{U}$ and $1/k\widetilde{\beta}$ in a way to satisfy the following $N+4$ relations: 
\begin{align}\label{eq2}
\begin{split}
&\int_{\mathbb{R}^3} \widetilde{Q}_i\, dp_i=0 \quad (i=1,\cdots,N), \qquad \sum_{i=1}^{N}\int_{\mathbb{R}^3} p_i^\mu \widetilde{Q}_i\,dp_i=0,
\end{split}
\end{align}
yielding the conservation of the partial particle four-flow and energy-momentum tensor of gas mixtures:
$$
\frac{\partial N_i^\mu}{\partial x^\mu}=0,\qquad \frac{\partial T^{\mu\nu}}{\partial x^\nu}=0.
$$
 The remaining part of this section is devoted to proving that the equations \eqref{eq2} uniquely determine the equilibrium coefficients $\widetilde{\mu}_i,\widetilde{U}^\mu$ and $\widetilde{\beta}$, which means our model is well-defined.
\begin{proposition}\label{JF}
Let $f_i\equiv f_i(p_i^\mu)\ge 0$ $(i=1,\cdots,N)$ be integrable and not identically zero in the almost-everywhere sense so that $N_i^\mu$ and $T_i^{\mu\nu}$ exist.   Then,	the system of equations (\ref{eq2}) leads to	
	\begin{enumerate}
	\item The presentation of the head part of the  J\"{u}ttner distribution using the moments of $f_i$: 
	\begin{equation*}
	\frac{g_{si}}{h^3}\exp\left(\widetilde{\beta}\widetilde{\mu}_i\right)=\frac{\int_{\mathbb{R}^3}f_i \,\frac{dp_i}{p_i^0}}{\int_{\mathbb{R}^3} e^{-\widetilde{\beta}\widetilde{U}^\mu p_{i\mu}} \,\frac{dp_i}{p_i^0}}.
	\end{equation*}
	\item The nonlinear relation satisfied by $\widetilde{\beta}$:
	\begin{align*}
	\sum_{i=1}^{N}\frac{m_i}{\tau_i}\frac{\int_{\mathbb{R}^3} e^{-c\widetilde{\beta}p_i^0}\,dp_i}{\int_{\mathbb{R}^3} e^{-c\widetilde{\beta}p_i^0}\frac{dp_i}{p_i^0}}\int_{\mathbb{R}^3}f_i \,\frac{dp_i}{p_i^0}=\frac{1}{c}\left[\left( \sum_{i=1}^{N}\frac{m_i }{\tau_i}n_iU_i^\mu\right)\left( \sum_{j=1}^{N}\frac{m_j }{\tau_j}n_jU_{j\mu}\right)\right]^{\frac{1}{2}}.
	\end{align*}
	\item The presentation of the equilibrium coefficient $\widetilde{U}^\mu$ using the macroscopic density $n_i$ and the Eckart four-velocity $U_i^{\mu}$:
	\begin{align*}
	\widetilde{U}^\mu=c\frac{\sum_{i=1}^{N}\frac{m_i}{\tau_i}n_iU_i^\mu}{\left[\left( \sum_{i=1}^{N}\frac{m_i }{\tau_i}n_iU_i^\mu\right)\left( \sum_{j=1}^{N}\frac{m_j }{\tau_j}n_jU_{j\mu}\right)\right]^{\frac{1}{2}}}.
	\end{align*}
	\end{enumerate}
\end{proposition}
	\begin{proof}
		\noindent (1) Substituting the relaxation operator $\widetilde{Q}_i$ of \eqref{Marle} into \eqref{eq2}, we have
		\begin{align*} 
		\int_{\mathbb{R}^3}\mathcal{J}_i \,\frac{dp_i}{p_i^0}&=\int_{\mathbb{R}^3}f_i \,\frac{dp_i}{p_i^0},\qquad
		\sum_{i=1}^{N}\frac{m_i}{\tau_i }\int_{\mathbb{R}^3}p^\mu_i\mathcal{J}_i \,\frac{dp_i}{p_i^0}=\sum_{i=1}^{N}\frac{m_i}{\tau_i }\int_{\mathbb{R}^3}p^\mu_if_i \,\frac{dp_i}{p_i^0}
		\end{align*}
		which, together with \eqref{eckart} and \eqref{Juttner}, yields
		\begin{align}\label{constraint}\begin{split} 
		\frac{g_{si}}{h^3}e^{\widetilde{\beta}\widetilde{\mu}_i}\int_{\mathbb{R}^3} e^{-\widetilde{\beta}\widetilde{U}^\mu p_{i\mu}} \,\frac{dp_i}{p_i^0}&=\int_{\mathbb{R}^3}f_i \,\frac{dp_i}{p_i^0},\cr 
		\sum_{i=1}^{N}\frac{m_i}{\tau_i }\frac{g_{si}}{h^3}e^{\widetilde{\beta}\widetilde{\mu}_i}\int_{\mathbb{R}^3}p^\mu_i e^{-\widetilde{\beta}\widetilde{U}^\mu p_{i\mu}} \,\frac{dp_i}{p_i^0}&=\sum_{i=1}^{N}\frac{m_i}{c\tau_i }n_iU^\mu_i		\end{split}\end{align}		
		for $i=1,\cdots, N$. 
		From the first identity of \eqref{constraint}, one finds
		\begin{equation}\label{gh}
		\frac{g_{si}}{h^3}e^{\widetilde{\beta}\widetilde{\mu}_i}=\frac{\int_{\mathbb{R}^3}f_i \,\frac{dp_i}{p_i^0}}{\int_{\mathbb{R}^3} e^{-\widetilde{\beta}\widetilde{U}^\mu p_{i\mu}} \,\frac{dp_i}{p_i^0}}.
		\end{equation}
		
		\noindent (2) We insert \eqref{gh} into the second identity of \eqref{constraint} to get
		\begin{equation}\label{constraint2}
		\sum_{i=1}^{N}\frac{m_i}{\tau_i }\frac{\int_{\mathbb{R}^3}p^\mu_i e^{-\widetilde{\beta}\widetilde{U}^\mu p_{i\mu}} \,\frac{dp_i}{p_i^0}}{\int_{\mathbb{R}^3} e^{-\widetilde{\beta}\widetilde{U}^\mu p_{i\mu}} \,\frac{dp_i}{p_i^0}}\int_{\mathbb{R}^3}f_i \,\frac{dp_i}{p_i^0}=\sum_{i=1}^{N}\frac{m_i}{c\tau_i }n_iU^\mu_i.
		\end{equation}
		To simplify the left-hand side of \eqref{constraint2}, we consider the following Lorentz transformation $\Lambda$ \cite{HRY,Strain2}
		\begin{align}\label{Lorentz transform}
		\Lambda=
		\begin{bmatrix}
		c^{-1}\widetilde{U}^0 & -c^{-1}\widetilde{U}^1 & -c^{-1}\widetilde{U}^2 & -c^{-1}\widetilde{U}^3 \cr
		-\widetilde{U}^1&  1+(\widetilde{U}^0-1)\frac{(\widetilde{U}^1)^2}{|\widetilde{U}|^2}&(\widetilde{U}^0-1)\frac{\widetilde{U}^1\widetilde{U}^2}{|\widetilde{U}|^2}  &(\widetilde{U}^0-1)\frac{\widetilde{U}^1\widetilde{U}^3}{|\widetilde{U}|^2}  \cr
		-\widetilde{U}^2& (\widetilde{U}^0-1)\frac{\widetilde{U}^1\widetilde{U}^2}{|\widetilde{U}|^2} &  1+(\widetilde{U}^0-1)\frac{(\widetilde{U}^2)^2 }{|\widetilde{U}|^2}&(\widetilde{U}^0-1)\frac{\widetilde{U}^2\widetilde{U}^3}{|\widetilde{U}|^2}  \cr
		-\widetilde{U}^3&  (\widetilde{U}^0-1)\frac{\widetilde{U}^1\widetilde{U}^3}{|\widetilde{U}|^2}& (\widetilde{U}^0-1)\frac{\widetilde{U}^2\widetilde{U}^3}{|\widetilde{U}|^2} &  1+(\widetilde{U}^0-1)\frac{(\widetilde{U}^3)^2}{|\widetilde{U}|^2}
		\end{bmatrix}
		\end{align}
		which maps $\widetilde{U}^\mu$ into the local rest frame $(c,0,0,0)$. Thus, it holds that
		\begin{equation}\label{c000}
		\Lambda \widetilde U^{\mu}=(c,0,0,0).
		\end{equation}
		 We set $P_i^\mu=\Lambda p_i^\mu$ and recall that the Minkowski inner product is invariant under the Lorentz transformation $\Lambda$ to see
		$$
		c^2m_i^2= p_i^\mu p_{i\mu}= \Lambda p_i^\mu \Lambda p_{i\mu}=P^\mu_i P_{i\mu}=\left(P_i^0\right)^2-\left|P_i\right|^2
		$$
		and hence the energy variable of $P_i^\mu$ is given by 
		\begin{equation}\label{P^0}
		P^0_i=\sqrt{(cm_i)^2+\left|P_i\right|^2}.
		\end{equation}
		With this in mind, we make the change of variables $P_i^\mu=\Lambda p_i^\mu$ to \eqref{constraint2} so that
		\begin{equation}\label{constraint3}
		\sum_{i=1}^{N}\frac{m_i}{\tau_i }\frac{\Lambda^{-1}\int_{\mathbb{R}^3}P^\mu_i e^{-c\widetilde{\beta}P_i^0} \,\frac{dP_i}{P_i^0}}{\int_{\mathbb{R}^3} e^{-c\widetilde{\beta}P_i^0} \,\frac{dP_i}{P_i^0}}\int_{\mathbb{R}^3}f_i \,\frac{dp_i}{p_i^0}=\sum_{i=1}^{N}\frac{m_i}{c\tau_i }n_iU^\mu_i.
		\end{equation}
		We used the fact that the volume element $dp_i/p_i^0$ are invariant under $\Lambda$ and
		\[
		\widetilde{U}^\mu p_{i\mu}=\Lambda\widetilde{U}^\mu \Lambda p_{i\mu}=\Lambda\widetilde{U}^\mu P_{i\mu}=cP^0_i,
		\]
		where the last identity follows from \eqref{c000}.
		Now, we observe from \eqref{P^0} that
		\begin{align*}
		\int_{\mathbb{R}^3}P^\mu_i e^{-c\widetilde{\beta}P_i^0} \,\frac{dP_i}{P_i^0}&=\left(\int_{\mathbb{R}^3} e^{-c\widetilde{\beta}\sqrt{(cm_i)^2+\left|P_i\right|^2}} \,dp_i,\int_{\mathbb{R}^3}P_i e^{-c\widetilde{\beta}\sqrt{(cm_i)^2+\left|P_i\right|^2}} \,\frac{dP_i}{\sqrt{(cm_i)^2+|P_i|^2}}\right)\cr
		&=\left(\int_{\mathbb{R}^3} e^{-c\widetilde{\beta}\sqrt{(cm_i)^2+\left|P_i\right|^2}} \,dP_i,0\right)
		\end{align*}
		which gives
		\begin{equation*}
		\Lambda^{-1}\int_{\mathbb{R}^3}P^\mu_i e^{-c\widetilde{\beta}P_i^0} \,\frac{dP_i}{P_i^0}=\left(\frac{1}{c}\int_{\mathbb{R}^3} e^{-c\widetilde{\beta}\sqrt{(cm_i)^2+\left|P_i\right|^2}} \,dP_i \right) \widetilde{U}^\mu.
		\end{equation*}
		Inserting this into \eqref{constraint3}, we get
		\begin{equation}\label{constraint4}
		\widetilde{U}^\mu=\left(\sum_{i=1}^{N}\frac{m_i}{\tau_i }\frac{\int_{\mathbb{R}^3} e^{-c\widetilde{\beta}p_i^0} \,dp_i }{\int_{\mathbb{R}^3} e^{-c\widetilde{\beta}p_i^0} \,\frac{dp_i}{p_i^0}}\int_{\mathbb{R}^3}f_i \,\frac{dp_i}{p_i^0}\right)^{-1}\sum_{i=1}^{N}\frac{m_i}{\tau_i }n_iU^\mu_i,
		\end{equation}
		and hence
		\begin{align*}
		c^2= \widetilde{U}^\mu \widetilde{U}_{\mu}&=\left(\sum_{i=1}^{N}\frac{m_i}{\tau_i }\frac{\int_{\mathbb{R}^3} e^{-c\widetilde{\beta}p_i^0} \,dp_i }{\int_{\mathbb{R}^3} e^{-c\widetilde{\beta}p_i^0} \,\frac{dp_i}{p_i^0}}\int_{\mathbb{R}^3}f_i \,\frac{dp_i}{p_i^0}\right)^{-2}\left(\sum_{i=1}^{N}\frac{m_i}{\tau_i }n_iU^\mu_i\right)\left( \sum_{i=1}^{N}\frac{m_i}{\tau_i }n_iU_{i\mu}\right).
		\end{align*}
		Therefore, we find the following identity:
		\begin{align}\label{relation for beta}
		\sum_{i=1}^{N}\frac{m_i}{\tau_i}\frac{\int_{\mathbb{R}^3} e^{-c\widetilde{\beta}p_i^0}\,dp_i}{\int_{\mathbb{R}^3} e^{-c\widetilde{\beta}p_i^0}\frac{dp_i}{p_i^0}}\int_{\mathbb{R}^3}f_i \,\frac{dp_i}{p_i^0}=\frac{1}{c}\left[\left(\sum_{i=1}^{N}\frac{m_i}{\tau_i }n_iU^\mu_i\right)\left( \sum_{i=1}^{N}\frac{m_i}{\tau_i }n_iU_{i\mu}\right)\right]^{1/2}.
		\end{align}
		
		\noindent (3) Finally, we combine \eqref{relation for beta} and \eqref{constraint4} to obtain the explicit form of $\widetilde{U}^\mu$:
		\begin{align*}
		\widetilde{U}^\mu=c\frac{\sum_{i=1}^{N}\frac{m_i}{\tau_i}n_iU_i^\mu}{\left[\left(\sum_{i=1}^{N}\frac{m_i}{\tau_i }n_iU^\mu_i\right)\left( \sum_{i=1}^{N}\frac{m_i}{\tau_i }n_iU_{i\mu}\right)\right]^{1/2}}.
		\end{align*}
		\end{proof}
	We now show that we can find a unique $\widetilde{\beta}$ satisfying the relation in Proposition \ref{JF} (2).
	\begin{proposition}\label{JF2} Under the same assumption in Proposition \ref{JF}, there exists a unique $\widetilde{\beta}$ that satisfies the nonlinear relation in Proposition \ref{JF} (2).
	\end{proposition}
	
	\begin{proof}
		For brevity, we  denote
		$$
		M_i(\widetilde{\beta})=\int_{\mathbb{R}^3} e^{-c\widetilde{\beta}p_i^0}\,dp_i,\qquad  \widetilde{M}_i(\widetilde{\beta})=\int_{\mathbb{R}^3} e^{-c\widetilde{\beta}p_i^0}\frac{dp_i}{p_i^0}.
		$$  
		From this notation, we rewrite the nonlinear relation in Proposition \ref{JF} (2) as
		\begin{align}\label{beta2}
		\sum_{i=1}^{N}\frac{m_i}{\tau_i}\frac{ M_i(\widetilde{\beta})}{ \widetilde{M}_i(\widetilde{\beta})}\int_{\mathbb{R}^3}f_i \,\frac{dp_i}{p_i^0}=\frac{1}{c}\left[\left(\sum_{i=1}^{N}\frac{m_i}{\tau_i }n_iU^\mu_i\right)\left( \sum_{i=1}^{N}\frac{m_i}{\tau_i }n_iU_{i\mu}\right)\right]^{1/2}.
		\end{align}
		Since
		$$
		\frac{d}{d\widetilde{\beta}} M_i(\widetilde{\beta})=-c\int_{\mathbb{R}^3} p_i^0e^{-c\widetilde{\beta}p_i^0}\,dp_i,\qquad \frac{d}{d\widetilde{\beta}} \widetilde{M}_i(\widetilde{\beta})=-c\int_{\mathbb{R}^3} e^{-c\widetilde{\beta}p_i^0}\,dp_i,
		$$
		we have from the H\"{o}lder inequality that
		\begin{align*}
		\frac{d}{d\widetilde{\beta}}\left\{\frac{ M_i(\widetilde{\beta})}{ \widetilde{M}_i(\widetilde{\beta})}  \right\}&=c\frac{\left(\int_{\mathbb{R}^3} e^{-c\widetilde{\beta}p_i^0}\,dp_i\right)^2-\int_{\mathbb{R}^3} p_i^0e^{-c\widetilde{\beta}p_i^0}\,dp_i\int_{\mathbb{R}^3} e^{-c\widetilde{\beta}p_i^0}\,\frac{dp_i}{p_i^0}}{\left(\widetilde{M}_i(\widetilde{\beta})\right)^2}
		< 0
		\end{align*}
		which says that $M_i/\widetilde{M}_i$ is strictly decreasing on $\widetilde{\beta}\in (0,\infty)$ for all $i=1,\cdots,N$. Moreover, $M_i/\widetilde{M}_i$ goes to $\infty$ as $\widetilde{\beta}$ approaches $0$ since
		\begin{align*}
		\frac{ M_i(\widetilde{\beta})}{ \widetilde{M}_i(\widetilde{\beta})}&=\frac{\int_{0}^\infty r^2 e^{-c\widetilde{\beta}\sqrt{(cm_i)^2+r^2}}\,dr}{\int_{0}^\infty \frac{r^2}{\sqrt{(cm_i)^2+r^2}} e^{-c\widetilde{\beta}\sqrt{(cm_i)^2+r^2}}\,dr}\cr
		&=\frac{1}{c\widetilde{\beta}}\frac{\int_{0}^\infty \frac{(cm_i)^2+2r^2}{\sqrt{(cm_i)^2+r^2}} e^{-c\widetilde{\beta}\sqrt{(cm_i)^2+r^2}}\,dr}{\int_{0}^\infty \frac{r^2}{\sqrt{(cm_i)^2+r^2}} e^{-c\widetilde{\beta}\sqrt{(cm_i)^2+r^2}}\,dr}\cr
		&\ge \frac{1}{c\widetilde{\beta}}
		\end{align*}
		where we used spherical coordinates and integration by parts. In a similar way, one can find that $M_i/\widetilde{M}_i\rightarrow cm_i$ as $\widetilde{\beta}\rightarrow \infty$:
		\begin{align*}
		cm_i<\frac{ M_i(\widetilde{\beta})}{ \widetilde{M}_i(\widetilde{\beta})}&=\frac{\int_{0}^\infty r^2 e^{-c\widetilde{\beta}\sqrt{(cm_i)^2+r^2}}\,dr}{\int_{0}^\infty \frac{r^2}{\sqrt{(cm_i)^2+r^2}} e^{-c\widetilde{\beta}\sqrt{(cm_i)^2+r^2}}\,dr}\cr
		&= 2\frac{\int_{0}^\infty \frac{r^2}{\sqrt{(cm_i)^2+r^2}} e^{-c\widetilde{\beta}\sqrt{(cm_i)^2+r^2}}\,dr}{\int_{0}^\infty  e^{-c\widetilde{\beta}\sqrt{(cm_i)^2+r^2}}\,dr}+\frac{\int_{0}^\infty \frac{(cm_i)^2}{\sqrt{(cm_i)^2+r^2}} e^{-c\widetilde{\beta}\sqrt{(cm_i)^2+r^2}}\,dr}{\int_{0}^\infty  e^{-c\widetilde{\beta}\sqrt{(cm_i)^2+r^2}}\,dr}\cr
		&\le \frac{2}{c\widetilde{\beta}}+cm_i.
		\end{align*}
		From these observations, we conclude that the left-hand side of \eqref{beta2} is strictly decreasing on $\widetilde{\beta}\in(0,\infty)$ and its range is 
		\begin{equation*}
		\left(\sum_{i=1}^{N}\frac{c(m_i)^2}{\tau_i}\int_{\mathbb{R}^3}f_i \,\frac{dp_i}{p_i^0},~\infty\right).
		\end{equation*} 
		Next, we show that the range of the r.h.s. of \eqref{beta2} lies in this open interval. For this, we observe from the Cauchy-Schwartz inequality that 
		$$
		p_i^\mu U_{i\mu}=\sqrt{(cm_i)^2+|p_i|^2}\sqrt{c^2+|U_i|^2}-p_i\cdot U_i \ge m_ic^2,
		$$
		and
		\begin{equation*}
				U_i^\mu U_{j \mu}=\sqrt{c^2+|U_i|^2}\sqrt{c^2+|U_j|^2} -U_i\cdot U_j\ge c^2,
		\end{equation*}
		to  bound the right-hand side of \eqref{beta2} from below as follows:
		\begin{align*}
		\frac{1}{c}\left[\left(\sum_{i=1}^{N}\frac{m_i}{\tau_i }n_iU^\mu_i\right)\left( \sum_{i=1}^{N}\frac{m_i}{\tau_i }n_iU_{i\mu}\right)\right]^{1/2}
		&=\frac{1}{c}\left[\sum_{i,j=1}^N(2-\delta_{ij})\frac{m_im_j}{\tau_i\tau_j}n_in_jU_i^\mu U_{j\mu}\right]^{1/2}\\
		&\geq\left[\sum_{i,j=1}^N(2-\delta_{ij})\frac{m_im_j}{\tau_i\tau_j}n_in_j\right]^{1/2}\\
		&=\sum_{i=1}^N\frac{m_i }{\tau_i}n_i\cr
		&=\sum_{i=1}^N\frac{m_i }{c\tau_i}\int_{\mathbb{R}^3}p_i^\mu U_{i\mu} f_i\,\frac{dp_i}{p_i^0}\cr
		&>\sum_{i=1}^N\frac{c(m_i)^2 }{\tau_i}\int_{\mathbb{R}^3} f_i\,\frac{dp_i}{p_i^0}.
		\end{align*}
	In the last line, we used the fact that $f_i$ is a integrable function so that $f_i$ does not concentrate at $p_i=m_iU_i$. Otherwise 
	  the strict inequality in the last line is not guaranteed. 
		Therefore,  we conclude that there is a one-to-one correspondence which guarantees the unique existence of $\widetilde{\beta}$ satisfying the nonlinear relation given in
		Proposition \ref{JF} (2).		
	\end{proof}
The following is the main result of this section, which follows directly from Proposition \ref{JF} and Proposition \ref{JF2}:
	\begin{theorem}\label{JF3} Under the same assumption in Proposition \ref{JF},
	the system of equations (\ref{eq2}) uniquely determines the J\"{u}ttner distribution $\mathcal{J}_i$ by	
	\begin{align*}
	\mathcal{J}_i=\frac{\int_{\mathbb{R}^3} f_i \frac{dp_i}{p_i^0}}{\int_{\mathbb{R}^3} e^{-c\widetilde{\beta} p_i^0} \frac{dp_i}{p_i^0}}\exp\left(-\widetilde{\beta}\widetilde{U}^\mu p_{i\mu}\right)
	\end{align*}
	for $i=1,\cdots,N$. Here the equilibrium parameter $\widetilde{U}^\mu$ is given in Proposition \ref{JF} (3) and $\widetilde{\beta}$ is determined 
	by the nonlinear relation in Proposition \ref{JF} (2).	
	\end{theorem}

	Before finishing this section, we recall the J\"{u}ttner distribution $\mathcal{J}$ for the Marle type relativistic BGK model for single-component gases \cite{BCNS, HY}:
	\begin{equation}\label{single}
	\mathcal{J}=\frac{\int_{\mathbb{R}^3}f \,\frac{dp}{p^0}}{\int_{\mathbb{R}^3} e^{- \frac{\beta}{mc} p^0}\,\frac{dp}{p^0}}e^{-\frac{\beta}{mc^2} U^\mu p_{\mu}},
	\end{equation}
	where $\beta$ is determined by the relation
	\begin{equation*} 
	\frac{\int_{\mathbb{R}^3} e^{-\frac{\beta}{mc} p^0}\,dp}{\int_{\mathbb{R}^3} e^{-\frac{\beta}{mc} p^0}\,\frac{dp}{p^0}}=\frac{n}{\int_{\mathbb{R}^3}f \,\frac{dp}{p^0}}.
	\end{equation*}
	Using the relation  $\beta=mc^2/kT$ and the modified Bessel function of the second kind:
	\begin{equation*} 
	K_i(\beta)=\int_0^\infty \cosh(ir)e^{-\beta \cosh(r)}\,dr,
	\end{equation*}
	we can rewrite \eqref{single}  as follows:
	\begin{align}\label{single J}\begin{split}
	\mathcal{J}= \frac{n}{4\pi m^2ckTK_2(\beta)}e^{-\frac{\beta}{mc^2} U^\mu p_{\mu}},
	\end{split}	\end{align}
	which is also widely used in relativistic kinetic theory. In the multi-species case, however, it turns out that the corresponding expression of $\mathcal{J}_i$ in the spirit of \eqref{single J} is extremely complicated:
	$$
	\mathcal{J}_i=\frac{\tau_i}{m_i}\frac{\frac{1}{c}\left[\left( \sum_{i=1}^{N}\frac{m_i }{\tau_i}n_iU_i^\mu\right)\left( \sum_{j=1}^{N}\frac{m_j }{\tau_j}n_jU_{j\mu}\right)\right]^{\frac{1}{2}}-\sum_{j\neq i}^{N}\frac{m_j}{\tau_j}\frac{\int_{\mathbb{R}^3} e^{-c\widetilde{\beta}p_j^0}\,dp_j}{\int_{\mathbb{R}^3} e^{-c\widetilde{\beta}p_j^0}\,\frac{dp_j}{p_j^0}}\int_{\mathbb{R}^3}f_j \,\frac{dp_j}{p_j^0} }{\int_{\mathbb{R}^3} e^{-c\widetilde{\beta}p_i^0}\,dp_i }e^{-\frac{\widetilde{U}^\mu p_{i \mu }}{k\widetilde{T}}}.
	$$
	which does not seem to have any advantage over the much simpler representation given in Theorem \ref{JF3}.

\section{Properties} 
In this section, we show that our relaxation model defined in the previous section satisfies all the crucial properties of the Boltzmann equation presented in Section 3.
\begin{theorem} Under the same assumption in Proposition \ref{JF}, our model \eqref{Marle} satisfies the following properties:
\begin{enumerate}

\item {\bf Conservation laws} The following relations hold:
\begin{align*}
	\begin{split}
		&\int_{\mathbb{R}^3} Q_i\, dp_i=0 \quad (i=1,\cdots,N) \qquad \sum_{i=1}^{N}\int_{\mathbb{R}^3} p_i^\mu Q_i\,dp_i=0.
	\end{split}
\end{align*}
Thus, the partial particle four-flow and the energy-momentum tensor of gas mixtures are conserved:
$$
\frac{\partial N_i^\mu}{\partial x^\mu}=0,\qquad \frac{\partial T^{\mu\nu}}{\partial x^\nu}=0.
$$
\item {\bf Equilibria} In global equilibrium, all distributions $f_i$ take the form of the J\"{u}ttner distribution with the same four-velocity and temperature.
\item {\bf H-theorem} The following inequality is satisfied:
\begin{align*}
	\sum_{i=1}^N \int_{\mathbb{R}^3} \widetilde{Q}_i\ln f_i \,dp_i\leq0.
\end{align*}
Therefore, we obtain the H-theorem:
$$
\frac{\partial S^\mu}{\partial x^\mu} \ge 0.
$$
\item {\bf Indifferentiability principle} If all constituents have the same
mass $m$ and the same relaxation time $\tau$, then the total momentum distribution $f:=\sum_{i=1}^N f_i$ is governed by the Marle model which is the relativistic BGK model for single-species gases \cite{Marle}.

 \end{enumerate}
\end{theorem}
The remaining part of this section is devoted to the proof of these properties.

\subsection{Conservation laws}
Recall that the equilibrium coefficients of our model are determined in a way that the following conservation laws hold:
\begin{equation}\label{cancellation} 
\frac{cm_i}{\tau_i}\int_{\mathbb{R}^3} \mathcal{J}_i-f_i \,\frac{dp_i}{p_i^0}=0,\qquad \sum_{i=1}^N \frac{cm_i}{\tau_i}\int_{\mathbb{R}^3} p_i^\mu \left(\mathcal{J}_i-f_i\right) \,\frac{dp_i}{p_i^0}=0.
\end{equation}
Multiplying \eqref{Marle} by $(1,p_i^\mu)$ and integrating over $p_i\in\mathbb{R}^3$, the identities in \eqref{cancellation} lead to the following relations
 	\begin{align*}
 	\int_{\mathbb{R}^3} \partial_{t} f_i+\frac{cp_i}{p_i^0}\cdot \nabla_x f_i\,dp_i=0,\qquad \sum_{i=1}^N\int_{\mathbb{R}^3}p_i^\mu\Big( \partial_{t} f_i+\frac{cp_i}{p_i^0}\cdot \nabla_x f_i\Big)\,dp_i=0.
 \end{align*}
Recalling the definition of $N_i^\mu$ and $T^{\mu\nu}$, this implies the desired conservation laws:
$$
\frac{\partial N_i^\mu}{\partial x^\mu}=0,\qquad \frac{\partial T^{\mu\nu}}{\partial x^\nu}=0,\qquad\forall i=1,\cdots, N.
$$

\subsection{Equilibria} To find the global equilibrium, assume $\widetilde{Q}_i\equiv0$ which yields
\begin{equation}\label{equilibria}
	c\int_{\mathbb{R}^3} \left(\mathcal{J}_i-f_i \right)\,\frac{dp_i}{p_i^0}=0,\qquad c\int_{\mathbb{R}^3} p_i^\mu \left(\mathcal{J}_i-f_i\right) \,\frac{dp_i}{p_i^0}=0,
\end{equation}
for $i=1,\cdots,N$.
Observe from \eqref{eckart} that
$$
c\int_{\mathbb{R}^3} p_i^\mu  f_i  \,\frac{dp_i}{p_i^0}=n_iU_i^\mu.
$$
Using the Lorentz transformation \eqref{Lorentz transform}, one finds
\begin{align*}
	c\int_{\mathbb{R}^3} p_i^\mu  \mathcal{J}_i  \,\frac{dp_i}{p_i^0}&=c\frac{\int_{\mathbb{R}^3} f_i \frac{dp_i}{p_i^0}}{\int_{\mathbb{R}^3} e^{-c\widetilde{\beta} p_i^0} \frac{dp_i}{p_i^0}}\int_{\mathbb{R}^3} p_i^\mu  e^{-\widetilde{\beta}\widetilde{U}^\mu p_{i\mu}}  \,\frac{dp_i}{p_i^0}\cr
	&=c\frac{\int_{\mathbb{R}^3} f_i \frac{dp_i}{p_i^0}}{\int_{\mathbb{R}^3} e^{-c\widetilde{\beta} p_i^0} \frac{dp_i}{p_i^0}}\Lambda^{-1}\int_{\mathbb{R}^3} p_i^\mu  e^{-c\widetilde{\beta} p_{i}^0}  \,\frac{dp_i}{p_i^0}.
\end{align*}
Since the spatial parts $(\mu=1,2,3)$ in the integral vanish due to the oddness, it reduces to
\begin{align*}
	c\int_{\mathbb{R}^3} p_i^\mu  \mathcal{J}_i  \,\frac{dp_i}{p_i^0}&=\frac{\int_{\mathbb{R}^3} f_i \frac{dp_i}{p_i^0}}{\int_{\mathbb{R}^3} e^{-c\widetilde{\beta} p_i^0} \frac{dp_i}{p_i^0}}\int_{\mathbb{R}^3}  e^{-c\widetilde{\beta} p_{i}^0}  \, dp_i \Lambda^{-1}(c,0,0,0)\cr
	&=\frac{\int_{\mathbb{R}^3} f_i \frac{dp_i}{p_i^0}}{\int_{\mathbb{R}^3} e^{-c\widetilde{\beta} p_i^0} \frac{dp_i}{p_i^0}}\int_{\mathbb{R}^3}  e^{-c\widetilde{\beta} p_{i}^0}  \, dp_i \widetilde{U}^\mu.
\end{align*}
From these observations, we see that the second identity of  \eqref{equilibria} is written as
$$
n_iU_i^\mu=\frac{\int_{\mathbb{R}^3} f_i \frac{dp_i}{p_i^0}}{\int_{\mathbb{R}^3} e^{-c\widetilde{\beta} p_i^0} \frac{dp_i}{p_i^0}}\int_{\mathbb{R}^3}  e^{-c\widetilde{\beta} p_{i}^0}  \, dp_i \widetilde{U}^\mu.
$$
Note that $U_i^\mu$ and $\widetilde{U}^\mu$ have a constant length in the following sense:
$$
U_i^\mu U_{i\mu}=c^2=\widetilde{U}^\mu\widetilde{U}_\mu,
$$
which gives
\begin{equation}\label{conclusion}
	n_i=\frac{\int_{\mathbb{R}^3} f_i \frac{dp_i}{p_i^0}}{\int_{\mathbb{R}^3} e^{-c\widetilde{\beta} p_i^0} \frac{dp_i}{p_i^0}}\int_{\mathbb{R}^3}  e^{-c\widetilde{\beta} p_{i}^0}  \, dp_i,\quad \text{and}\quad U_i^\mu= \widetilde{U}^\mu.
\end{equation}
The second identity of \eqref{conclusion} means that at global equilibrium, all of the four-velocities for $i$ species become identical to $\widetilde{U}^\mu$. To investigate the temperatures for $i$ species, we apply the change of variables $p_i/(m_ic)\rightarrow p_i$ and spherical coordinates to the first identity of \eqref{conclusion} to obtain
\begin{align}\label{temperature}\begin{split}
		\frac{1}{n_i}\int_{\mathbb{R}^3} f_i \frac{dp_i}{p_i^0}&=\frac{\int_{\mathbb{R}^3} e^{-c\widetilde{\beta} p_i^0} \frac{dp_i}{p_i^0}}{\int_{\mathbb{R}^3}  e^{-c\widetilde{\beta} p_{i}^0}  \, dp_i}\cr
		&=\frac{1}{m_ic}\frac{\int_{\mathbb{R}^3} e^{-m_ic^2\widetilde{\beta} \sqrt{1+|p_i|^2}} \frac{dp_i}{\sqrt{1+|p_i|^2}}}{\int_{\mathbb{R}^3}  e^{-m_ic^2\widetilde{\beta} \sqrt{1+|p_i|^2}}  \, dp_i}\cr
		&=\frac{1}{m_ic}\frac{K_1(m_ic^2\widetilde{\beta})}{K_2(m_ic^2\widetilde{\beta})}
\end{split}\end{align}
where  $K_i$ denotes the modified Bessel function of the second kind:
\begin{align}\label{bessel function}
	K_1(\beta)=\int_{0}^\infty e^{-\beta\sqrt{1+|r|^2}}\,dr,\qquad K_2(\beta)=\int_{0}^\infty \frac{2r^2+1}{\sqrt{1+r^2}} e^{-\beta\sqrt{1+|r|^2}}\,dr.
\end{align}
Recall from \cite[eq. (8.4)]{Kremer} that in the Marle's formulation for a single component gas, the ratio $\beta_i:=m_ic^2/kT_i$ is determined by the means of constraints \eqref{equilibria} as follows
$$
\frac{m_ic}{n_i}\int_{\mathbb{R}^3} f_i \frac{dp_i}{p_i^0}=\frac{K_1(\beta_i)}{K_2(\beta_i)},
$$
which combined with \eqref{temperature} gives
$$
\frac{K_1}{K_2}\left(\frac{m_ic^2}{kT_i}\right)=\frac{K_1}{K_2}\left(\frac{m_ic^2}{k\widetilde{T}}\right).
$$
Since it was shown in \cite[Appendix]{BCNS} that $K_1/K_2$ is  strictly monotone, we get 
$$
T_i= \widetilde{T},\qquad i=1,\cdots, N.
$$
At global equilibrium, we thus conclude that for all constituent $i$, the momentum distributions $f_i$  become J\"{u}ttner distributions $\mathcal{J}_i$ sharing the common four-velocity $\widetilde{U}^\mu$ and temperature $\widetilde{T}$.
 
\subsection{H-theorem}  From \eqref{cancellation} we have
\begin{align*}
\sum_{i=1}^N  \int_{\mathbb{R}^3} \widetilde{Q}_i \ln \mathcal{J}_i\, dp_i&=\sum_{i=1}^N  \ln \left(\frac{\int_{\mathbb{R}^3} f_i \frac{dp_i}{p_i^0}}{\int_{\mathbb{R}^3} e^{-c\widetilde{\beta} p_i^0} \frac{dp_i}{p_i^0}}\right)\int_{\mathbb{R}^3} \widetilde{Q}_i \, dp_i-\sum_{i=1}^N \widetilde{\beta}\widetilde{U}_\mu \int_{\mathbb{R}^3} p_i^{\mu} \widetilde{Q}_i  \, dp_i\cr
&=0.
\end{align*}
From this observation, one finds
	\begin{align*}
	\sum_{i=1}^N \int_{\mathbb{R}^3} \widetilde{Q}_i  \ln \left( \frac{f_i h^3}{g_{s_i}}\right) \,dp_i&=\sum_{i=1}^N \int_{\mathbb{R}^3} \widetilde{Q}_i  \ln f_i \,dp_i\cr
	&=  		 		\sum_{i=1}^N \int_{\mathbb{R}^3} \widetilde{Q}_i  \left(\ln f_i -\ln \mathcal{J}_i \right) \,dp_i \cr
	&=   \sum_{i=1}^N \frac{cm_i}{\tau_i}\int_{\mathbb{R}^3} \mathcal{J}_i\left(1-\frac{f_i}{\mathcal{J}_i}\right)  \ln \frac{f_i}{ \mathcal{J}_i} \,\frac{dp_i}{p_i^0}\cr
	&\le 0
	\end{align*}
thanks to the following elementary inequality: 
	$$
	(1-x)\ln x \le 0,\qquad \forall x>0.
	$$
Therefore, the momentum distribution $f_i$ of \eqref{Marle} satisfies
	$$
	\sum_{i=1}^N \partial_t\int_{\mathbb{R}^3} f_i \ln \left( \frac{f_i h^3}{g_{s_i}}\right) \,dp_i+	c\sum_{i=1}^N \text{div} _x\int_{\mathbb{R}^3}p_i f_i \ln \left( \frac{f_i h^3}{g_{s_i}}\right) \,\frac{dp_i}{p_i^0}\le 0,
	$$
i.e. the following entropy inequality holds:
$$
\frac{\partial S^\mu}{\partial x^\mu} = -k\sum_{i=1}^N\partial_t\int_{\mathbb{R}^3}  f_i \ln \left( \frac{f_i h^3}{g_{s_i}}\right) \, dp_i -kc\sum_{i=1}^N\text{div}_x\int_{\mathbb{R}^3}p_i  f_i \ln \left( \frac{f_i h^3}{g_{s_i}}\right) \,\frac{dp_i}{p_i^0}  \ge 0.
$$

\subsection{Indifferentiability principle} Let
\begin{equation}\label{indiff0}
m_1=m_2=\cdots=m_N\ (=m),\qquad \tau_1=\tau_2=\cdots=\tau_N\ (=\tau),
\end{equation}
then all the four-momentum $p_i^\mu$ reduce to
$$
p^\mu=\left(\sqrt{(cm)^2+|p|^2},p\right).
$$
Adding the equation \eqref{Marle} over $i=1,\cdots,N$, we get   
$$
\partial_t f+\frac{cp}{p^0}\cdot\nabla_xf=\frac{cm}{\tau p^0}\Big(\sum_{i=1}^N\mathcal{J}_i-f\Big)
$$
where $f$ denotes the total momentum distribution
$$
f(t,x,p)=\sum_{i=1}^N f_i(t,x,p).
$$
By definition, we see that
\begin{equation}\label{indiff}
\sum_{i=1}^N\mathcal{J}_i=\sum_{i=1}^N\frac{\int_{\mathbb{R}^3}f_i\,\frac{dp_i}{p_i^0}}{\int_{\mathbb{R}^3}e^{-c\widetilde{\beta} p_i^0}\,\frac{dp_i}{p_i^0}}e^{-\widetilde{\beta} \widetilde{U}^\mu p_{i \mu }}=\frac{ \int_{\mathbb{R}^3}f\,\frac{dp}{p^0}}{\int_{\mathbb{R}^3}e^{-c\widetilde{\beta} p^0}\,\frac{dp}{p^0}}e^{-\widetilde{\beta} \widetilde{U}^\mu p_{\mu}}.
\end{equation}	
Recalling the definition of $n_iU_i^\mu$, it follows from \eqref{indiff0} that 
\begin{equation}\label{four-flow indiff}
\sum_{i=1}^{N}n_iU_i^\mu= \sum_{i=1}^Nc\int_{\mathbb{R}^3}p^\mu_if_i \,\frac{dp_i}{p_i^0}=c\int_{\mathbb{R}^3}p^\mu f \,\frac{dp}{p^0}\equiv  nU^\mu,
\end{equation}
which gives
\begin{align}\label{indiff22}\begin{split}
\widetilde{U}^\mu&=c\frac{\sum_{i=1}^{N}\frac{m_i}{\tau_i}n_iU_i^\mu}{\sqrt{\left( \sum_{i=1}^{N}\frac{m_i }{\tau_i}n_iU_i^\mu\right)\left( \sum_{j=1}^{N}\frac{m_j }{\tau_j}n_jU_{j\mu}\right)}}\cr
&=c\frac{ nU^\mu}{\sqrt{\left(nU^\mu\right)\left( nU_\mu\right)}}\cr
&=U^\mu.
\end{split}\end{align}
In the last line, we used the fact that $U^\mu$ has a constant length $U^\mu U_\mu=c^2$. Recall that the equilibrium coefficient $\widetilde{\beta}$ is determined by the relation in
Proposition \ref{JF} (2). 
Using the notation $\beta=mc^2\widetilde{\beta} $ and applying \eqref{indiff0} and \eqref{four-flow indiff} to the above identity, one finds
\begin{equation}\label{indiff3}
\frac{\int_{\mathbb{R}^3} e^{-\frac{\beta}{mc} p^0}\,dp}{\int_{\mathbb{R}^3} e^{-\frac{\beta}{mc} p^0}\,\frac{dp}{p^0}}\int_{\mathbb{R}^3}f \,\frac{dp}{p^0}=\frac{1}{c}\left[\left( \sum_{i=1}^{N}n_iU_i^\mu\right)\left( \sum_{j=1}^{N}n_jU_{j\mu}\right)\right]^{\frac{1}{2}}=n.
\end{equation}
Going back to \eqref{indiff} with \eqref{indiff22} and \eqref{indiff3}, we get
$$
\sum_{i=1}^N\mathcal{J}_i=\frac{\int_{\mathbb{R}^3}f \,\frac{dp}{p^0}}{\int_{\mathbb{R}^3} e^{-\frac{\beta}{mc} p^0}\,\frac{dp}{p^0}}e^{-\frac{\beta}{mc^2} U^\mu p_{\mu}}\equiv \mathcal{J}
$$
which is exactly the same as the J\"{u}ttner distribution for a single-component gas given in \eqref{single J}. In summary, we show that under assumption \eqref{indiff0}, the total momentum distribution function of \eqref{Marle} is governed by Marle model:  
 \begin{align*}
\partial_t f+\frac{cp}{p^0}\cdot\nabla_xf&=\frac{cm}{\tau p^0}\left(\mathcal{J} -f\right).
\end{align*}

\section{Newtonian limit}

In this section, we show that in the Newtonian limit, our model \eqref{Marle} reduces to the BGK model for gas mixtures of classical particles suggested by  Groppi, Rjasanow, and Spiga \cite{GRS,GRS2}:
\begin{equation}\label{Bisi}
\partial_t f_i  + v\cdot\nabla_{x}f_i=\nu_i\left(\mathcal{M}_i(f) -f_i \right),\qquad i=1,\cdots, N,
\end{equation}
where $f_i\equiv f_i(t,x,v)$ is the velocity distribution function representing the number density of particles on the phase point $(x,v)\in \Omega\times \mathbb{R}^3$ at time $t\in\mathbb{R}_+$, and $\nu_i$ is the collision frequency. The classical local Maxwellian $\mathcal{M}_i(f)$ is given by
\begin{equation}\label{Maxwellian}
\mathcal{M}_i(f)=\mathcal{M}_i(f_1,\cdots,f_N)=n_i^{\text{nr}}\left(\frac{m_i}{2\pi T^{\text{nr}}}\right)^{3/2}e^{-m_i|v-U^{\text{nr}}|^2/2T^{\text{nr}}},
\end{equation}
where the classical macroscopic parameters are given by the following manners:
\begin{itemize}
	\item Classical macroscopic number density, bulk velocity, and temperature for each $i$ species gas, $n_i^{\text{nr}},u_i^{\text{nr}}$ and $T_i^{\text{nr}}$ are defined by
	\begin{align*}
		n_i^{\text{nr}}&=\int_{\mathbb{R}^3} f_i\, dv, \qquad u_i^{\text{nr}}=\frac{1}{n_i^{\text{nr}}}\int_{\mathbb{R}^3} v f_i \, d v, \qquad T_i^{\text{nr}}=\frac{m_i}{3n_i^{\text{nr}}}\int_{\mathbb{R}^3}\left| v-u_i^{\text{nr}}\right|^2 f_i\, d v.
	\end{align*}
	\item The equilibrium velocity $U^{\text{nr}}$ and the equilibrium temperature $T^{\text{nr}}$ are defined as follows:
	\begin{align}\label{UT}
		\begin{split}
			U^{\text{nr}}&=\frac{\sum_{i=1}^N \nu_i m_i n_i^{\text{nr}} u_i^{\text{nr}}}{\sum_{i=1}^N \nu_i m_i n_i^{\text{nr}}},\\
			\frac{3}{2}\left(\sum_{i=1}^N\nu_i n_i^{\text{nr}}\right)T^{\text{nr}}&=\sum_i^N \nu_i \left[\frac{1}{2}m_in_i^{\text{nr}}\left(|u_i^{\text{nr}}|^2-|U^{\text{nr}}|^2\right)+\frac{3}{2}n_i^{\text{nr}}T_i^{\text{nr}} \right].
		\end{split}
	\end{align}
\end{itemize}

To investigate the Newtonian limit of \eqref{Marle}, we follow the line of  \cite[Section 3]{BCNS}, in which the Marle model for a single-component gas is considered. We extract the dimensionless numbers from variables $t,x,p_i$ as
$$
t=\overline{t} s,\qquad x=\overline{x}L,\qquad p_i= v\mu_i
$$
Here, $s$, $L$, $\mu_i$ are the typical time, the typical length, and the typical microscopic momentum respectively, and  $\overline{t}$ , $\overline{x}$, $v$ stand for the corresponding dimensionless numbers. Then the momentum distribution $f(t,x,p_i)$ can be represented as
$$
f_i(t,x,p_i)=\frac{\mathcal{N}}{\mu_i^3}\overline{f}_i(\overline{t},\overline{x}, v)
$$
where $\mathcal{N}$ is the typical number density, and 
\begin{equation*}
\mu_i=\frac{m_i L}{s}.
\end{equation*}
 Now we substitute the dimensionless numbers into  \eqref{Marle} to obtain 
\begin{equation}\label{dimensionless Marle2}
\frac{\partial}{\partial\overline{t}}\overline{f}_i  +\frac{ 1}{ \sqrt{1+|\varepsilon v|^2}} v\cdot\nabla_{\overline{x}}\overline{f}_i=  \frac{\nu_i}{\sqrt{1+|\varepsilon v|^2}}\left(\frac{\mu_i^3}{\mathcal{N}}\mathcal{J}_i(f)-\overline{f}_i \right)
\end{equation}
where $\varepsilon$ and $\nu_i$ denote
$$
\varepsilon:=\frac{\mu_i}{cm_i}=\frac{L}{cs},\qquad \nu_i:=\frac{s}{\tau_i}
$$
Note that $\varepsilon=\mu_i/cm_i$ is indenpendent of the index $i$. Our aim is to show that \eqref{dimensionless Marle2} reduces to the classical BGK model \eqref{Bisi} as $\varepsilon\rightarrow 0$. We first investigate the dimensionless form of $\mathcal{J}_i$ in the following lemma.
\begin{lemma}\label{dimensionless J} Assume $\varepsilon \ll 1$. Then,
	$$
	\mathcal{J}_i(f)=\frac{\mathcal{N}}{\mu_i^3}\overline{\mathcal{J}}_i(\overline{f})
	$$
	where $\overline{\mathcal{J}}_i(\overline{f})$ denotes
	$$
	\overline{\mathcal{J}}_i(\overline{f})=\frac{\int_{\mathbb{R}^3} \overline{f}_i \frac{d v}{\sqrt{1+|\varepsilon  v|^2}}}{\int_{\mathbb{R}^3} e^{-m_ic^2\widetilde{\beta} \sqrt{1+|\varepsilon  v|^2}} \frac{d v}{\sqrt{1+|\varepsilon  v|^2}}}\exp\left\{-m_ic^2\widetilde{\beta}\left(\overline{U}^0\sqrt{1+|\varepsilon v|^2} - \varepsilon\overline{U}\cdot v \right)\right\}
	$$
and $\overline{U}^\mu$ is
	\begin{align}\label{dimensionless U}\begin{split}
		\overline{U}^\mu=\frac{\sum_{i=1}^{N}\frac{m_i}{\tau_i}\left(\int_{\mathbb{R}^3}\overline{f}_i\,dv,\varepsilon \int_{\mathbb{R}^3} v\overline{f}_i\frac{dv}{\sqrt{1+|\varepsilon  v|^2}} \right)}{\sqrt{\left(\sum_{i=1}^{N}\frac{m_i }{\tau_i}\int_{\mathbb{R}^3}\overline{f}_i\,dv\right)^2-\left|\sum_{i=1}^{N}\frac{m_i }{\tau_i}\varepsilon\int_{\mathbb{R}^3} v\overline{f}_i\frac{dv}{\sqrt{1+|\varepsilon  v|^2}} \right|^2}}.	
\end{split}\end{align}	
	.
\end{lemma}
\begin{proof}
	For reader's convenience, we record the definition of $\mathcal{J}_i(f)$
	\begin{align*}
	\mathcal{J}_i(f)=\frac{\int_{\mathbb{R}^3} f_i \frac{dp_i}{p_i^0}}{\int_{\mathbb{R}^3} e^{-c\widetilde{\beta} p_i^0} \frac{dp_i}{p_i^0}}e^{-\widetilde{\beta}\widetilde{U}^\mu p_{i \mu }}.
	\end{align*}
	Here $\widetilde{U}^\mu$ is defined as
	$$
	\widetilde{U}^\mu=c\frac{\sum_{i=1}^{N}\frac{m_i}{\tau_i}n_iU_i^\mu}{\sqrt{\left( \sum_{i=1}^{N}\frac{m_i }{\tau_i}n_iU_i^\mu\right)\left( \sum_{j=1}^{N}\frac{m_j }{\tau_j}n_jU_{j\mu}\right)}}
	$$
	where $n_i$ and $n_iU_i^\mu$ are given as
	\begin{align*}
	n_i&=\Biggl\{\biggl(\int_{\mathbb{R}^3}f_i \,dp_i\biggl)^2-\sum_{k=1}^3\biggl(\int_{\mathbb{R}^3}p^k_if_i \,\frac{dp_i}{p_i^0}\biggl)^2\Biggl\}^{1/2},\cr 
	n_iU_i^\mu&=c\int_{\mathbb{R}^3}p_i^\mu f_i\frac{dp_i}{p_i^0},
	\end{align*}
	and  $\widetilde{\beta}$ is uniquely determined by the relation in Proposition 3.1 (2).
	$$
	\sum_{i=1}^{N}\frac{m_i}{\tau_i}\frac{\int_{\mathbb{R}^3} e^{-c\widetilde{\beta}p_i^0}\,dp_i}{\int_{\mathbb{R}^3} e^{-c\widetilde{\beta}p_i^0}\,\frac{dp_i}{p_i^0}}\int_{\mathbb{R}^3}f_i \,\frac{dp_i}{p_i^0}=\frac{1}{c}\bigg[\biggl( \sum_{i=1}^{N}\frac{m_i }{\tau_i}n_iU_i^\mu\biggl)\biggl( \sum_{j=1}^{N}\frac{m_j }{\tau_j}n_jU_{j\mu}\biggl)\bigg]^{1/2}.
	$$
Inserting dimensionless numbers into the definition of $n_iU_i^\mu$, we have 
	$$
	n_iU_i^\mu= c\mathcal{N}\left(\int_{\mathbb{R}^3}\overline{f}_i\,dv, \varepsilon \int_{\mathbb{R}^3} v\overline{f}_i\frac{d v}{\sqrt{1+|\varepsilon v|^2}} \right),
	$$
	which yields
		\begin{align}\label{numerator}\begin{split}
	&\hspace{-0.7cm}\left( \sum_{i=1}^{N}\frac{m_i }{\tau_i}n_iU_i^\mu\right)\left( \sum_{j=1}^{N}\frac{m_j }{\tau_j}n_jU_{j\mu}\right)\cr
	&=\sum_{i,j}^{N}(2-\delta_{ij})\frac{m_i m_j}{\tau_i\tau_j}n_iU_i^\mu n_jU_{j\mu}\cr
	&=\sum_{i,j}^{N}(2-\delta_{ij})c^2\frac{m_i\mathcal{N} m_j\mathcal{N}}{\tau_i\tau_j}\cr
	&\times\left\{ \int_{\mathbb{R}^3}\overline{f}_i\,dv\int_{\mathbb{R}^3}\overline{f}_j\,dv-\varepsilon^2 \int_{\mathbb{R}^3} v\overline{f}_i\frac{dv}{\sqrt{1+|\varepsilon  v|^2}}\cdot\int_{\mathbb{R}^3} v\overline{f}_j\frac{dv}{\sqrt{1+|\varepsilon  v|^2}} \right\}\cr
	&=c^2\mathcal{N}^2\left(\sum_{i=1}^{N}\frac{m_i  }{\tau_i}\int_{\mathbb{R}^3}\overline{f}_i\,dv\right)^2-c^2\varepsilon^2\mathcal{N}^2\left|\sum_{i=1}^{N}\frac{m_i }{\tau_i} \int_{\mathbb{R}^3} v\overline{f}_i\frac{dv}{\sqrt{1+|\varepsilon  v|^2}} \right|^2.
	\end{split}\end{align}
	Here, $\delta_{ij}$ is the Kronecker delta. Thus we get
	\begin{align*} 
	\widetilde{U}^\mu&=c\frac{\sum_{i=1}^{N}\frac{m_i}{\tau_i} \left( \int_{\mathbb{R}^3}\overline{f}_i\,dv,\varepsilon  \int_{\mathbb{R}^3} v\overline{f}_i\frac{dv}{\sqrt{1+|\varepsilon  v|^2}} \right)}{\sqrt{\left(\sum_{i=1}^{N}\frac{m_i  }{\tau_i}\int_{\mathbb{R}^3}\overline{f}_i\,dv\right)^2-\left|\sum_{i=1}^{N}\frac{m_i  }{\tau_i}\varepsilon\int_{\mathbb{R}^3} v\overline{f}_i\frac{dv}{\sqrt{1+|\varepsilon  v|^2}} \right|^2}}
	= c\overline{U}^\mu.
	 \end{align*}
	Therefore, one finds
	\begin{align*}
	\mathcal{J}_i(f)&=\frac{\int_{\mathbb{R}^3} f_i \frac{dp_i}{p_i^0}}{\int_{\mathbb{R}^3} e^{-c\widetilde{\beta} p_i^0} \frac{dp_i}{p_i^0}}e^{-\widetilde{\beta}\widetilde{U}^\mu p_{i \mu }}\cr
	&=\frac{\mathcal{N}}{\mu_i^3}\frac{\int_{\mathbb{R}^3} \overline{f}_i \frac{d v}{\sqrt{1+|\varepsilon  v|^2}}}{\int_{\mathbb{R}^3} e^{-m_ic^2\widetilde{\beta} \sqrt{1+|\varepsilon  v|^2}} \frac{d v}{\sqrt{1+|\varepsilon  v|^2}}}\exp\left\{-m_ic^2\widetilde{\beta}\left(\overline{U}^0\sqrt{1+|\varepsilon v|^2} - \varepsilon\overline{U}\cdot v \right)\right\}\cr
	&=\frac{\mathcal{N}}{\mu_i^3}\overline{\mathcal{J}}_i(\overline{f}).
	\end{align*}
\end{proof}
Due to Lemma \ref{dimensionless J}, it only remains to show that $\overline{\mathcal{J}}_i(\overline{f})$ approaches the Maxwellian attractor \eqref{Maxwellian} as $\varepsilon\rightarrow 0$. For this, we present the following lemma.
\begin{lemma}\label{beta epsilon}
	For sufficiently small $\varepsilon>0$, we have
	\[
	\widetilde{\beta}^{-1}\leq C\varepsilon^2.
	\]
	for some $C>0$.
\end{lemma}
\begin{proof}	We divide the proof of the Lemma into the following two step:\\
\noindent{\bf Step I:} First, we start with a slightly weaker statement: $\displaystyle \lim_{\varepsilon\rightarrow 0}\widetilde{\beta}=\infty$.\newline 
\noindent Proof of the claim:
Observe that the dimensionless form of the left-hand side of the relation in Proposition 3.1 (2) reads
	\begin{align}\label{lhs}\begin{split}
	&\sum_{i=1}^{N}\frac{m_i}{\tau_i}\frac{\int_{\mathbb{R}^3} e^{-c\widetilde{\beta}p_i^0}\,dp_i}{\int_{\mathbb{R}^3} e^{-c\widetilde{\beta}p_i^0}\,\frac{dp_i}{p_i^0}}\int_{\mathbb{R}^3}f_i \,\frac{dp_i}{p_i^0}\cr
	&\qquad=\sum_{i=1}^{N}\frac{m_i\mathcal{N}}{\tau_i}\frac{\int_{\mathbb{R}^3} e^{-m_ic^2\widetilde{\beta}\sqrt{1+|\varepsilon  v|^2}}\,dv}{\int_{\mathbb{R}^3} e^{-m_ic^2\widetilde{\beta}\sqrt{1+|\varepsilon  v|^2}}\frac{dv}{\sqrt{1+|\varepsilon  v|^2}}}\int_{\mathbb{R}^3}\overline{f}_i \,\frac{dv}{\sqrt{1+|\varepsilon v|^2}}\cr
	&\qquad=\sum_{i=1}^{N}\frac{m_i\mathcal{N}}{\tau_i}\frac{\int_{\mathbb{R}^3} e^{-m_ic^2\widetilde{\beta}\sqrt{1+|\varepsilon  v|^2}}\,dv}{\int_{\mathbb{R}^3} e^{-m_ic^2\widetilde{\beta}\sqrt{1+|\varepsilon  v|^2}}\frac{dv}{\sqrt{1+|\varepsilon v|^2}}}\left(\int_{\mathbb{R}^3}\overline{f}_i \,dv+\mathcal{O}(\varepsilon^2)\right)
	\end{split}\end{align}
	where we used
	\begin{align}\label{identity1}\begin{split}
	\frac{1}{\sqrt{1+|\varepsilon v|^2}}&= 1+\mathcal{O}(\varepsilon^2).
	\end{split}\end{align}
In the case of the right-hand side of the relation in Proposition 3.1 (2), one can see from \eqref{numerator} that 
\begin{align}\label{rhs}\begin{split}
		&\frac{1}{c}\left[\biggl( \sum_{i=1}^{N}\frac{m_i }{\tau_i}n_iU_i^\mu\biggl)\biggl( \sum_{j=1}^{N}\frac{m_j }{\tau_j}n_jU_{j\mu}\biggl)\right]^{1/2}\cr
		&\qquad=\sum_{i=1}^{N}\frac{m_i\mathcal{N} }{\tau_i}\int_{\mathbb{R}^3}\overline{f}_i\,dv\left[1- \left|\varepsilon\frac{\sum_{i=1}^{N}\frac{m_i  }{\tau_i}\int_{\mathbb{R}^3} v\overline{f}_i\frac{dv}{\sqrt{1+|\varepsilon  v|^2}}}{\sum_{i=1}^{N}\frac{m_i  }{\tau_i}\int_{\mathbb{R}^3}\overline{f}_i\,dv} \right|^2\right]^{1/2}\cr
		&\qquad=\sum_{i=1}^{N}\frac{m_i\mathcal{N} }{\tau_i}\int_{\mathbb{R}^3}\overline{f}_i\,dv+\mathcal{O}(\varepsilon^2)
\end{split}\end{align}
where we used 
\begin{align}\label{identity0}
	\sqrt{1-|\varepsilon x|^2}= 1-\frac{\varepsilon^2|x|^2}{2}+\mathcal{O}(\varepsilon^4).
\end{align}
	Combining \eqref{lhs} and \eqref{rhs}, one finds
	\begin{align}\label{beta dimensionless}
		\begin{split}
	&\sum_{i=1}^{N}\frac{m_i\mathcal{N}}{\tau_i}\left(\frac{\int_{\mathbb{R}^3} e^{-m_ic^2\widetilde{\beta}\sqrt{1+|\varepsilon  v|^2}}\,dv}{\int_{\mathbb{R}^3} e^{-m_ic^2\widetilde{\beta}\sqrt{1+|\varepsilon  v|^2}}\frac{dv}{\sqrt{1+|\varepsilon  v|^2}}}-1\right)\int_{\mathbb{R}^3}\overline{f}_i \,dv\\
	&\qquad=\left(\sum_{i=1}^{N}\frac{m_i\mathcal{N}}{\tau_i}\frac{\int_{\mathbb{R}^3} e^{-m_ic^2\widetilde{\beta}\sqrt{1+|\varepsilon  v|^2}}\,dv}{\int_{\mathbb{R}^3} e^{-m_ic^2\widetilde{\beta}\sqrt{1+|\varepsilon  v|^2}}\frac{dv}{\sqrt{1+|\varepsilon  v|^2}}}\right)\mathcal{O}(\varepsilon^2)+\mathcal{O}(\varepsilon^2)
		\end{split}
	\end{align}
	which implies
	$$
	\frac{\int_{\mathbb{R}^3} e^{-m_ic^2\widetilde{\beta}\sqrt{1+|\varepsilon  v|^2}}\,dv}{\int_{\mathbb{R}^3} e^{-m_ic^2\widetilde{\beta}\sqrt{1+|\varepsilon  v|^2}}\frac{dv}{\sqrt{1+|\varepsilon v|^2}}} \rightarrow 1\qquad\text{as}\qquad \varepsilon\rightarrow 0.
	$$
	Since the above term is strictly decreasing on $\widetilde{\beta}\in(0,\infty)$ and ranges from $1$ to $\infty$ (see the proof of Proposition \ref{JF2}), one can see that $\widetilde{\beta}\rightarrow \infty$ as $\varepsilon\rightarrow 0$. \noindent\newline
	{\bf Step II: Proof of the lemma:} It follows from the change of variables $\varepsilon v\rightarrow  v$ and spherical coordinates	that 
	\begin{align}\label{MM}\begin{split}  
	\int_{\mathbb{R}^3} e^{-m_ic^2\widetilde{\beta}\sqrt{1+|\varepsilon  v|^2}}\frac{dv}{\sqrt{1+|\varepsilon  v|^2}}&=\frac{4\pi}{\varepsilon^3}\int_{0}^\infty \frac{r^2}{\sqrt{1+r^2}}e^{-\beta_i\sqrt{1+|r|^2}}\,dr,\cr
	\int_{\mathbb{R}^3} e^{-m_ic^2\widetilde{\beta}\sqrt{1+|\varepsilon  v|^2}}\,dv&= \frac{4\pi}{\varepsilon^3}\int_{0}^\infty r^2 e^{-\beta_i\sqrt{1+|r|^2}}\,dr,
\end{split}	\end{align}
	where $\beta_i$ denotes
	$$
	\beta_i=m_ic^2\widetilde{\beta}.
	$$
	Using an integration by parts and \eqref{bessel function}, \eqref{MM} reduces to
	\begin{align} \label{modified}\begin{split}
	\int_{\mathbb{R}^3} e^{-m_ic^2\widetilde{\beta}\sqrt{1+|\varepsilon  v|^2}}\frac{dv}{\sqrt{1+|\varepsilon  v|^2}}&= \frac{4\pi}{\varepsilon^3\beta_i} K_1(\beta_i),\cr
	\int_{\mathbb{R}^3} e^{-m_ic^2\widetilde{\beta}\sqrt{1+|\varepsilon  v|^2}}\,dv&=  \frac{4\pi}{\varepsilon^3\beta_i} K_2(\beta_i).
	\end{split}\end{align}
Inserting \eqref{modified} into \eqref{beta dimensionless}, one finds
\begin{equation*}
	\sum_{i=1}^{N}\frac{m_i\mathcal{N}}{\tau_i}\left(\frac{K_2(\beta_i)}{K_1(\beta_i)}-1\right)\int_{\mathbb{R}^3}\overline{f}_i \,dv=\mathcal{O}(\varepsilon^2).
\end{equation*}
Since $K_2\geq K_1$, and all the summands are positive, we can conclude from this that
\begin{align*}
	\frac{K_2(\beta_i)}{K_1(\beta_i)}-1=\mathcal{O}(\varepsilon^2),
\end{align*}
which means
\begin{align}\label{bessel} 
\left|\frac{K_2(\beta_i)}{K_1(\beta_i)}-1\right|\leq C\varepsilon^2,
\end{align}
for some $C>0$.
	It is well-known \cite{CK} that for $\beta_i\gg 1$, the modified Bessel functions of the second kind are expanded as
	\begin{align}\label{expansion}\begin{split}
	K_1(\beta_i)&=\sqrt{\frac{\pi}{2\beta_i}}\frac{1}{e^{\beta_i}}\left(1+\frac{3}{8\beta_i}-\frac{15}{2!(8\beta_i)^2}+\cdots\right),\cr
	K_2(\beta_i)&=\sqrt{\frac{\pi}{2\beta_i}}\frac{1}{e^{\beta_i}}\left(1+\frac{15}{8\beta_i}+\frac{105}{2!(8\beta_i)^2}+\cdots\right),
	\end{split}\end{align}
which gives 
\begin{align*}
	\frac{K_2(\beta_i)}{K_1(\beta_i)}=1+\frac{3}{2\beta_i}+\mathcal{O}\left(\frac{1}{\beta_i^2}\right).
\end{align*}
for some $C>0$.
That is, 
\begin{align}\label{com}
\left|\frac{K_2(\beta_i)}{K_1(\beta_i)}-1-\frac{3}{2\beta_i}\right|\leq C\left(\frac{1}{\beta_i^2}\right).
\end{align}
Combining \eqref{bessel} and \eqref{com}, we get
\[
\frac{3}{2\beta_i}\leq \left|\frac{K_2(\beta_i)}{K_1(\beta_i)}-1\right|+C\left(\frac{1}{\beta_i^2}\right)\leq C\varepsilon^2+C\left(\frac{1}{\beta_i^2}\right).
\]
If we solve this quadratic inequality w.r.t $1/\beta_i$, we get
\begin{equation}\label{case}
\frac{1}{\beta_i}\leq \frac{\frac{3}{2}-\sqrt{\frac{9}{4}-4
C^2\varepsilon^2}}{2C},\quad\mbox{ or }\quad
\frac{1}{\beta_i}\geq \frac{\frac{3}{2}+\sqrt{\frac{9}{4}-4
		C^2\varepsilon^2}}{2C}.
\end{equation}
In Step I, we've shown that $\beta_i$ is large when $\varepsilon$ is sufficiently small. Therefore, we consider only the first case in \eqref{case}:
\[
\frac{1}{\beta_i}\leq \frac{\frac{3}{2}-\sqrt{\frac{9}{4}-4
		C^2\varepsilon^2}}{2C}.
\]
Then, we use
\[
\frac{3}{2}-\sqrt{\frac{9}{4}-x}\leq x  \quad\mbox{ when }0\leq x\leq 2
\]
to find
\[
\frac{1}{\beta_i}\leq 2C\varepsilon^2.
\]
This completes the proof.
\end{proof}

\begin{proposition}
	 The dimensionless form of $\mathcal{J}_i$ approaches the Maxwellian \eqref{Maxwellian} in the Newtonian limit:
	 \begin{align*}
	 	\overline{\mathcal{J}}_i(\overline{f})\rightarrow \mathcal{M}_i(\overline{f})\qquad \text{as}\qquad \varepsilon\rightarrow0.
	 \end{align*}
\begin{proof}
Recall from Lemma \ref{dimensionless J} that
	$$
		\overline{\mathcal{J}}_i(\overline{f_i})=\frac{\int_{\mathbb{R}^3} \overline{f}_i \frac{d v}{\sqrt{1+|\varepsilon  v|^2}}}{\int_{\mathbb{R}^3} e^{-m_ic^2\widetilde{\beta} \sqrt{1+|\varepsilon  v|^2}} \frac{d v}{\sqrt{1+|\varepsilon  v|^2}}}\exp\left\{-m_ic^2\widetilde{\beta}\left(\overline{U}^0\sqrt{1+|\varepsilon v|^2} - \varepsilon\overline{U}\cdot v \right)\right\}.
	$$
It follows from \eqref{identity1} that
	$$
	\int_{\mathbb{R}^3} \overline{f}_i \frac{dv}{\sqrt{1+|\varepsilon  v|^2}}=\int_{\mathbb{R}^3} \overline{f}_i dv+\mathcal{O}(\varepsilon^2),
	$$
	which together with Lemma \ref{beta epsilon}, \eqref{modified} and \eqref{expansion} gives
	\begin{align}\label{multiplier}\begin{split}
	\frac{\int_{\mathbb{R}^3} \overline{f}_i \frac{dv}{\sqrt{1+|\varepsilon  v|^2}}}{\int_{\mathbb{R}^3} e^{-m_ic^2\widetilde{\beta} \sqrt{1+|\varepsilon  v|^2}}\frac{dv}{\sqrt{1+|\varepsilon  v|^2}}}
	&=\frac{\int_{\mathbb{R}^3} \overline{f}_i dv+\mathcal{O}(\varepsilon^2)}{\frac{4\pi}{\varepsilon^3\beta_i}\sqrt{\frac{\pi}{2\beta_i}}e^{-\beta_i}(1+\mathcal{O}(\beta_i^{-1}))}\cr
	&=\varepsilon^3\left(\frac{\beta_i}{2\pi}\right)^{3/2}e^{\beta_i}\left(\int_{\mathbb{R}^3} \overline{f}_i dv+\mathcal{O}(\varepsilon^2)\right)\\
	&=\varepsilon^3\left(\frac{\beta_i}{2\pi}\right)^{3/2}e^{\beta_i}\left(n_i^{\text{nr}}+\mathcal{O}(\varepsilon^2)\right)	
\end{split}	\end{align}
where the classical macroscopic number density $n_i^{\text{nr}}$ of $\overline{f_i}$ is given by
\begin{align*}
	n_i^{\text{nr}}=\int_{\mathbb{R}^3} \overline{f}_i dv.
\end{align*}
On the other hand, we recall from \eqref{dimensionless U} that 
	\begin{align*}
	\overline{U}^0\sqrt{1+|\varepsilon v|^2} - \varepsilon\overline{U}\cdot v &=\frac{\sqrt{1+|\varepsilon v|^2}\sum_{i=1}^{N}\frac{m_i  }{\tau_i}\int_{\mathbb{R}^3}\overline{f}_i\,dv-\varepsilon^2v\cdot\sum_{i=1}^{N}\frac{m_i }{\tau_i}\int_{\mathbb{R}^3} v\overline{f}_i\frac{dv}{\sqrt{1+|\varepsilon  v|^2}}}{\sqrt{\left(\sum_{i=1}^{N}\frac{m_i  }{\tau_i}\int_{\mathbb{R}^3}\overline{f}_i\,dv\right)^2-\left|\sum_{i=1}^{N}\frac{m_i  }{\tau_i}\varepsilon\int_{\mathbb{R}^3} v\overline{f}_i\frac{dv}{\sqrt{1+|\varepsilon  v|^2}} \right|^2}}
	\end{align*}
Applying \eqref{identity1} and \eqref{identity0} to the numerator, we get
	\begin{align*}
	&\sqrt{1+|\varepsilon v|^2}\sum_{i=1}^{N}\frac{m_i }{\tau_i}\int_{\mathbb{R}^3}\overline{f}_i\,dv-\varepsilon^2v\cdot\sum_{i=1}^{N}\frac{m_i }{\tau_i}\int_{\mathbb{R}^3} v\overline{f}_i\frac{dv}{\sqrt{1+|\varepsilon  v|^2}}\cr
	&=\sum_{i=1}^{N}\frac{m_i }{\tau_i} \int_{\mathbb{R}^3}\overline{f}_i\,dv\left(1+\frac{|\varepsilon v|^2}{2}+\mathcal{O}(\varepsilon^4)\right)-\varepsilon^2v\cdot\sum_{i=1}^{N}\frac{m_i }{\tau_i}\int_{\mathbb{R}^3} v\overline{f}_i \,dv \left(1+\mathcal{O}(\varepsilon^2)\right).
	\end{align*}
In the same manner, the denominator reduces to
	\begin{align}\label{denominator}\begin{split}
	&\sum_{i=1}^{N}\frac{m_i  }{\tau_i}\int_{\mathbb{R}^3}\overline{f}_i\,dv \left[1-\left|\varepsilon\frac{\sum_{i=1}^{N}\frac{m_i  }{\tau_i}\int_{\mathbb{R}^3} v\overline{f}_i\frac{dv}{\sqrt{1+|\varepsilon  v|^2}}}{ \sum_{i=1}^{N}\frac{m_i  }{\tau_i}\int_{\mathbb{R}^3}\overline{f}_i\,dv } \right|^2\right]^{1/2}\cr
	&\qquad=\sum_{i=1}^{N}\frac{m_i  }{\tau_i}\int_{\mathbb{R}^3}\overline{f}_i\,dv\left(1-\frac{\varepsilon^2}{2}\left|\frac{\sum_{i=1}^{N}\frac{m_i  }{\tau_i}\int_{\mathbb{R}^3} v\overline{f}_i\,dv}{ \sum_{i=1}^{N}\frac{m_i  }{\tau_i}\int_{\mathbb{R}^3}\overline{f}_i\,dv } \right|^2\right)+\mathcal{O}(\varepsilon^4).
	\end{split}\end{align}
	From these observations, we have
	\begin{align*}
&\overline{U}^0\sqrt{1+|\varepsilon v|^2} - \varepsilon\overline{U}\cdot v \cr
&=\frac{\sum_{i=1}^{N}\frac{m_i }{\tau_i}\int_{\mathbb{R}^3}\overline{f}_i\,dv\left(1+\frac{|\varepsilon v|^2}{2}+\mathcal{O}(\varepsilon^4)\right)-\varepsilon^2v\cdot\sum_{i=1}^{N}\frac{m_i }{\tau_i} \int_{\mathbb{R}^3} v\overline{f}_i \,dv \left(1+\mathcal{O}(\varepsilon^2)\right)}{\sum_{i=1}^{N}\frac{m_i  }{\tau_i}\int_{\mathbb{R}^3}\overline{f}_i\,dv\left(1-\frac{\varepsilon^2}{2}\left|\frac{\sum_{i=1}^{N}\frac{m_i  }{\tau_i}\int_{\mathbb{R}^3} v\overline{f}_i\,dv}{ \sum_{i=1}^{N}\frac{m_i  }{\tau_i}\int_{\mathbb{R}^3}\overline{f}_i\,dv } \right|^2\right)+\mathcal{O}(\varepsilon^4)} \cr
	&=\frac{1+\frac{|\varepsilon v|^2}{2}-\varepsilon^2v\cdot\frac{\sum_{i=1}^{N}\frac{m_i }{\tau_i}\int_{\mathbb{R}^3} v\overline{f}_i \,dv}{\sum_{i=1}^{N}\frac{m_i  }{\tau_i}\int_{\mathbb{R}^3}\overline{f}_i\,dv}+\mathcal{O}(\varepsilon^4) }{1-\frac{\varepsilon^2}{2}\left|\frac{\sum_{i=1}^{N}\frac{m_i  }{\tau_i}\int_{\mathbb{R}^3} v\overline{f}_i\,dv}{ \sum_{i=1}^{N}\frac{m_i  }{\tau_i}\int_{\mathbb{R}^3}\overline{f}_i\,dv } \right|^2 +\mathcal{O}(\varepsilon^4)} \cr
	&=\left(1+\frac{\varepsilon^2|v|^2}{2}-\varepsilon^2v\cdot\frac{\sum_{i=1}^{N}\frac{m_i }{\tau_i}\int_{\mathbb{R}^3}v\overline{f}_i \,dv}{\sum_{i=1}^{N}\frac{m_i  }{\tau_i}\int_{\mathbb{R}^3}\overline{f}_i\,dv}+\mathcal{O}(\varepsilon^4) \right)\left(1+\frac{\varepsilon^2}{2}\left|\frac{\sum_{i=1}^{N}\frac{m_i  }{\tau_i}\int_{\mathbb{R}^3} v\overline{f}_i\,dv}{ \sum_{i=1}^{N}\frac{m_i  }{\tau_i}\int_{\mathbb{R}^3}\overline{f}_i\,dv } \right|^2+\mathcal{O}(\varepsilon^4)\right)\cr
	&=1+\frac{\varepsilon^2}{2}\left|\frac{\sum_{i=1}^{N}\frac{m_i  }{\tau_i}\int_{\mathbb{R}^3} v\overline{f}_i\,dv}{ \sum_{i=1}^{N}\frac{m_i  }{\tau_i}\int_{\mathbb{R}^3}\overline{f}_i\,dv } \right|^2+\frac{\varepsilon^2|v|^2}{2}-\varepsilon^2v\cdot\frac{\sum_{i=1}^{N}\frac{m_i }{\tau_i}\int_{\mathbb{R}^3} v\overline{f}_i \,dv}{\sum_{i=1}^{N}\frac{m_i  }{\tau_i}\int_{\mathbb{R}^3}\overline{f}_i\,dv}+\mathcal{O}(\varepsilon^4)\cr
	&=1+\frac{\varepsilon^2}{2}\left|v-U^{\text{nr}}\right|^2+\mathcal{O}(\varepsilon^4)
	\end{align*}
	where $U^{\text{nr}}$ is given by (See \eqref{UT})
	\begin{align*}
		U^{\text{nr}}=\frac{\sum_{i=1}^{N}\frac{m_i  }{\tau_i}\int_{\mathbb{R}^3} v\overline{f}_i\,dv}{ \sum_{i=1}^{N}\frac{m_i  }{\tau_i}\int_{\mathbb{R}^3}\overline{f}_i\,dv }=\frac{\sum_{i=1}^{N}\nu_im_i\int_{\mathbb{R}^3} v\overline{f}_i\,dv}{ \sum_{i=1}^{N}\nu_im_i\int_{\mathbb{R}^3}\overline{f}_i\,dv }=\frac{\sum_{i=1}^{N}\nu_im_in_i^{\text{nr}}u_i^{\text{nr}}}{ \sum_{i=1}^{N}\nu_im_in_i^{\text{nr}} },
	\end{align*}
	with the classical macroscopic bulk velocity $u_i^{\text{nr}}$ of $\overline{f_i}$ given by
	\begin{align*}
		n_i^{\text{nr}}u_i^{\text{nr}}&=\int_{\mathbb{R}^3}v\overline{f_i}\,dv.
	\end{align*}
 Thus, we obtain
	\begin{align*}
	\exp\left\{-m_ic^2\widetilde{\beta}\left(\overline{U}^0\sqrt{1+|\varepsilon v|^2} - \varepsilon\overline{U}\cdot v \right)\right\}&=\exp\left\{-m_ic^2\widetilde{\beta}\left(1+\frac{\varepsilon^2}{2}\left|v-U\right|^2+\mathcal{O}(\varepsilon^4) \right)\right\}\cr
	&=\exp(-\beta_i)\exp\left(-\varepsilon^2m_ic^2\widetilde{\beta}\frac{\left|v-U\right|^2}{2}+\mathcal{O}(\varepsilon^2)\right),
	\end{align*}
	which, together with \eqref{multiplier}, gives
	\begin{align*}
		\begin{split}
	\overline{\mathcal{J}}_i(\overline{f_i})&=\left(\varepsilon^3\left(\frac{\beta_i}{2\pi}\right)^{3/2}e^{\beta_i}\left(n_i^{\text{nr}}+\mathcal{O}(\varepsilon^2)\right)	\right)\left(\exp(-\beta_i)\exp\left(-\varepsilon^2m_ic^2\widetilde{\beta}\frac{\left|v-U^{\text{nr}}\right|^2}{2}+\mathcal{O}(\varepsilon^4)\right)\right)\\
	&=\left(\frac{\varepsilon^2m_ic^2\widetilde{\beta}}{2\pi}\right)^{3/2}\left(n_i^{\text{nr}}+\mathcal{O}(\varepsilon^2)\right)\exp\left(-\varepsilon^2m_ic^2\widetilde{\beta}\frac{\left|v-U^{\text{nr}}\right|^2}{2}+\mathcal{O}(\varepsilon^4)\right).
		\end{split}
	\end{align*}
	It only remains to show that $\varepsilon^2c^2\widetilde{\beta}$ goes to $1/T^{\text{nr}}$ as $\varepsilon\rightarrow0$, where $T^{\text{nr}}$ is given by
	\begin{align*}
		T^{\text{nr}}=\sum_{i=1}^{N}\nu_i\left(\frac{1}{2} m_i n_i^{\text{nr}}\left(|u_i^{\text{nr}}|^2-|U^{\text{nr}}|^2\right)+\frac{3}{2} n_i^{\text{nr}}T_i^{\text{nr}}\right)\bigg/ \left(\frac{3}{2}\sum_{i=1}^{N}\nu_in_i^{\text{nr}} \right),
	\end{align*}
	with the actual macroscopic temperature $T_i^{\text{nr}}$ of $\overline{f_i}$ defined by
	\begin{align*}
		3n_i^{\text{nr}}T_i^{\text{nr}}&=m_i\int_{\mathbb{R}^3}|v-u_i^{\text{nr}}|^2\overline{f}_i \,dv.
	\end{align*}
	For this, we recall from \eqref{denominator} that the right-hand side of the relation in Proposition 3.1 (2) is expanded as
	\begin{align}\label{rhs2}
	\sum_{i=1}^{N}\frac{m_i\mathcal{N} }{\tau_i}\int_{\mathbb{R}^3}\overline{f}_i\,dv-\frac{\varepsilon^2}{2}\frac{\left|\sum_{i=1}^{N}\frac{m_i\mathcal{N} }{\tau_i}\int_{\mathbb{R}^3} v\overline{f}_i\,dv\right|^2}{ \sum_{i=1}^{N}\frac{m_i\mathcal{N} }{\tau_i}\int_{\mathbb{R}^3}\overline{f}_i\,dv }+\mathcal{O}(\varepsilon^4).
	\end{align}
	For the left-hand side of the relation in Proposition 3.1 (2), it follows from \eqref{lhs}, \eqref{modified} and \eqref{expansion} that
	\begin{align}\label{lhs2}
	\begin{split}
	&\sum_{i=1}^{N}\frac{m_i}{\tau_i}\frac{\int_{\mathbb{R}^3} e^{-c\widetilde{\beta}p_i^0}\,dp_i}{\int_{\mathbb{R}^3} e^{-c\widetilde{\beta}p_i^0}\,\frac{dp_i}{p_i^0}}\int_{\mathbb{R}^3}f_i \,\frac{dp_i}{p_i^0}\cr
	&=\sum_{i=1}^{N}\frac{m_i\mathcal{N}}{\tau_i}\left(1+\frac{3}{2\beta_i}+\mathcal{O}\left(\beta_i^{-2}\right)\right)\int_{\mathbb{R}^3}\overline{f}_i \,\frac{dv}{\sqrt{1+|\varepsilon v|^2}}\cr
	&=\sum_{i=1}^{N}\frac{m_i\mathcal{N}}{\tau_i}\left(1+\frac{3}{2\beta_i}+\mathcal{O}(\varepsilon^4)\right)\left(\int_{\mathbb{R}^3}\overline{f}_i \,dv-\frac{\varepsilon^2}{2}\int_{\mathbb{R}^3}|v|^2\overline{f}_i \,dv+\mathcal{O}(\varepsilon^4)\right),
	\end{split}
	\end{align}
	where we used Lemma \ref{beta epsilon} and the following expansion
	$$
	\frac{1}{\sqrt{1+|\varepsilon v|^2}}=1-\frac{\varepsilon^2|v|^2}{2}+\mathcal{O}(\varepsilon^4).
	$$
	Comparing both sides \eqref{rhs2} and \eqref{lhs2} of the relation in Proposition 3.1 (2), one finds
	\begin{align*}
	&\sum_{i=1}^{N}\frac{m_i\mathcal{N} }{\tau_i}\int_{\mathbb{R}^3}\overline{f}_i\,dv-\frac{\varepsilon^2}{2}\frac{\left|\sum_{i=1}^{N}\frac{m_i\mathcal{N} }{\tau_i}\int_{\mathbb{R}^3} v\overline{f}_i\,dv\right|^2}{ \sum_{i=1}^{N}\frac{m_i\mathcal{N} }{\tau_i}\int_{\mathbb{R}^3}\overline{f}_i\,dv }\cr
	&\qquad=\sum_{i=1}^{N}\frac{m_i\mathcal{N}}{\tau_i}\left(\int_{\mathbb{R}^3}\overline{f}_i \,dv+\frac{3}{2\beta_i}\int_{\mathbb{R}^3}\overline{f}_i \,dv-\frac{\varepsilon^2}{2}\int_{\mathbb{R}^3}|v|^2\overline{f}_i \,dv\right) +\mathcal{O}(\varepsilon^4)
	\end{align*}
	where $\varepsilon^2/\beta_i$ was absorbed into $\mathcal{O}(\varepsilon^4)$ by Lemma \ref{beta epsilon}. Using the relation $\nu_i=s/\tau_i$ and $\beta_i=m_ic^2\widetilde{\beta}$,	we get
	\begin{align*}
	&\frac{3}{2}\left(\sum_{i=1}^{N}\nu_i\int_{\mathbb{R}^3}\overline{f}_i \,dv \right)\frac{1}{c^2\widetilde{\beta}}\\
	&\qquad=\sum_{i=1}^{N} \nu_im_i\frac{\varepsilon^2}{2}\int_{\mathbb{R}^3}|v|^2\overline{f}_i \,dv-\frac{\varepsilon^2}{2}\frac{\left|\sum_{i=1}^{N} \nu_im_i \int_{\mathbb{R}^3} v\overline{f}_i\,dv\right|^2}{ \sum_{i=1}^{N}\nu_i m_i\int_{\mathbb{R}^3}\overline{f}_i\,dv }+\mathcal{O}(\varepsilon^4)\\
	&\qquad=\sum_{i=1}^{N} \nu_im_i\frac{\varepsilon^2}{2}\int_{\mathbb{R}^3}|v|^2\overline{f}_i \,dv-\frac{\varepsilon^2}{2}\sum_{i=1}^{N} \nu_im_i n_i|U^{\text{nr}}|^2+\mathcal{O}(\varepsilon^4)
	\end{align*}
	and hence,
	\begin{align*}
		\frac{1}{\varepsilon^2c^2\widetilde{\beta}}=\frac{\sum_{i=1}^{N} \nu_im_i\frac{1}{2}\int_{\mathbb{R}^3}|v|^2\overline{f}_i \,dv-\frac{1}{2}\sum_{i=1}^{N} \nu_im_i n_i^{\text{nr}}|U^{\text{nr}}|^2}{\frac{3}{2}\sum_{i=1}^{N}\nu_im_i\int_{\mathbb{R}^3}\overline{f}_i \,dv}+\mathcal{O}(\varepsilon^2)
	\end{align*}
	Finally, using the fact that
	\begin{align*}
	\frac{3}{2}\sum_{i=1}^{N}\nu_i n_i^{\text{nr}}T_i^{\text{nr}}&=\frac{1}{2}\sum_{i=1}^{N} \nu_im_i\int_{\mathbb{R}^3}|v-u_i^{\text{nr}}|^2\overline{f}_i \,dv\cr
	&=\frac{1}{2}\sum_{i=1}^{N} \nu_im_i\left(\int_{\mathbb{R}^3}|v|^2\overline{f}_i \,dv-n_i^{\text{nr}}|u_i^{\text{nr}}|^2\right),
	\end{align*}
	we obtain from \eqref{UT}
	\begin{align*}
		\frac{1}{\varepsilon^2c^2\widetilde{\beta}}&=\frac{\sum_{i=1}^{N}\nu_i\left(\frac{1}{2} m_i n_i^{\text{nr}}\left(|u_i^{\text{nr}}|^2-|U^{\text{nr}}|^2\right)+\frac{3}{2} n_i^{\text{nr}}T_i^{\text{nr}}\right)}{\frac{3}{2}\sum_{i=1}^{N}\nu_im_in_i^{\text{nr}}}+\mathcal{O}(\varepsilon^2)\\
		&=T^{\text{nr}}+\mathcal{O}(\varepsilon^2)
	\end{align*}
	which completes the proof.
	
	
\end{proof}
\end{proposition}

\noindent{\bf Data availability:} No data was used for the research described in the article.\newline

\noindent{\bf Conflicts of interest:} The authors declare that they have no conflicts of interest. \newline

\section*{Acknowledgments}
 B.-H. Hwang was supported  by Basic Science Research Program through the National Research Foundation of Korea(NRF) funded by the Ministry of Education (No. NRF-2019R1A6A1A10073079).  S.-B. Yun was supported  by Basic Science Research Program through the National Research Foundation of Korea(NRF) funded by the Ministry of Education (No. NRF-2023R1A2C1005737).

\end{document}